\documentclass{article}
\usepackage[utf8]{inputenc}
\usepackage{amsfonts,amsmath,amssymb,amsthm}
\usepackage{mathtools,bm}
\usepackage{algorithm,algpseudocode}
\usepackage{xcolor}
\usepackage{caption}
\usepackage{subcaption}
\usepackage{geometry}
\newcommand{\R}{\mathbb{R}}
\newcommand{\E}{\mathbb{E}}
\newcommand{\PP}{\mathbb{P}}
\newcommand{\QQ}{\mathbb{Q}}
\newcommand{\ACal}{\mathcal{A}}
\newcommand{\BCal}{\mathcal{B}}
\newcommand{\CCal}{\mathcal{C}}
\newcommand{\DCal}{\mathcal{D}}

\newcommand{\FCal}{\mathcal{F}}
\newcommand{\LCal}{\mathcal{L}}
\newcommand{\MCal}{\mathcal{M}}
\newcommand{\NCal}{\mathcal{N}}
\newcommand{\PCal}{\mathcal{P}}
\newcommand{\RCal}{\mathcal{R}}
\newcommand{\XCal}{\mathcal{X}}
\newcommand{\YCal}{\mathcal{Y}}
\newcommand{\GCal}{\mathcal{G}}
\newcommand{\Mmf}{\mathfrak{M}}

\newtheorem{assumption}{Assumption}

\newtheorem{lemma}{Lemma}
\newtheorem{theorem}{Theorem}
\newtheorem{definition}{Definition}
\newtheorem{corollary}{Corollary}
\newtheorem{remark}{Remark}

\title{Distributionally Time-Varying Online Stochastic Optimization under Polyak-\L{}ojasiewicz Condition with Application in Conditional Value-at-Risk Statistical Learning\footnote{This work is supported in part by the Australian Research Council under the Discovery Project DP210102454 and the Australian Government, via grant AUSMURIB000001 associated with ONR MURI grant N00014-19-1-2571.}}
\newcommand*\samethanks[1][\value{footnote}]{\footnotemark[#1]}
\author{Yuen-Man Pun\thanks{CIICADA Lab, School of Engineering, The Australian National University; Email: \{yuenman.pun, iman.shames\}@anu.edu.au}
\and Farhad Farokhi\thanks{Department of Electrical and Electronic Engineering, The University of Melbourne; Email: ffarokhi@unimelb.edu.au}
\and Iman Shames\samethanks[2]
}

\begin{document}

\maketitle

\begin{abstract}
In this work, we consider a sequence of stochastic optimization problems following a time-varying distribution via the lens of online optimization. Assuming that the loss function satisfies the Polyak-\L{}ojasiewicz condition, we apply online stochastic gradient descent and establish its dynamic regret bound that is composed of cumulative distribution drifts and cumulative gradient biases caused by stochasticity. The distribution metric we adopt here is Wasserstein distance, which is well-defined without the absolute continuity assumption or with a time-varying support set. We also establish a regret bound of online stochastic proximal gradient descent when the objective function is regularized. Moreover, we show that the above framework can be applied to the Conditional Value-at-Risk (CVaR) learning problem. Particularly, we improve an existing proof on the discovery of the PL condition of the CVaR problem, resulting in a regret bound of online stochastic gradient descent.
\end{abstract}

\section{Introduction}
In a stochastic optimization problem, one aims to make a decision by minimizing the expectation of a loss function following an unknown distribution, which can be approximated via sampling. As many problems in real world involve uncertain parameters, stochastic programming has been extensively applied to almost all areas of science and engineering~\cite{SDR21}, such as telecommunication~\cite{Gaivor05}, finance~\cite{TVZ08,ZV14}, and marketing~\cite{Pflug12}, just to name a few. Most works in stochastic programming study scenarios when the underlying distribution is stationary, which, nevertheless, may not apply to problems in dynamic environments. Examples include problems in finance and sociology where the expansion of economy and the evolution of demographics can significantly modify the underlying distributions. Another example is a source localization problem of a substance leakage or mitigating its effect, where the distribution of the substance changes in the space due to movements of the source, diffusion, or changes in the environment.

A na\"{i}ve approach to solving the problem is to find a solution of a worst-case scenario of a set of distributions that contains the whole trajectory of the underlying distribution over time and use tools from the distributionally robust optimization (DRO) to solve it. DRO, which has been proved to be of extreme importance in machine learning~\cite{AEK15,SND17,KESN19,faro20,MRY13}, focuses on finding the solution of a worst-case scenario of a set of distributions (often known as ambiguity set) constructed near the empirical distribution and assumed to contain the true distribution~\cite{CW10,EI06,EK18,BS04}; also see~\cite{PDM16,BDD+13,DY10,GS10,HH13,JG16,PW07,woza12} for different constructions of ambiguity sets. However, the solution in DRO is known to be very conservative, especially when the ambiguity set is large. As the underlying distribution may drift significantly over time and making the ambiguity set large, this approach may not be desirable by applying one solution to all possible distributions in the ambiguity set.

Another approach is to view it as a sequence of stochastic optimization problems following a time-varying distribution over different time steps. This fits into an online optimization framework \cite{MSJ+16,DBT+19,ASD19,HK14}, in which a decision maker makes a series of decision based on the observations at previous rounds. Recently, there have been works that interplay between online optimization and stochastic programming; see, for example,~\cite{SF20,KMD22,CDH23,WBD21,WZ20,CZP21,JLZ20}. However, as far as we concerned, these works mostly consider sequences of convex loss functions, which may not be applicable to applications with nonconvex losses. Moreover, most works quantify the distribution change using the distance between optimal solutions at consecutive time steps, which is less intuitive as it involves the behavior of the loss function.

Motivated by the above discussion, we consider a sequence of expectation loss minimization problems that satisfy the Polyak-\L{}ojasiewicz (PL) condition. This class of functions, albeit not necessarily convex, satisfies certain quadratic growth condition, which is shown to be exhibited in a number of optimization problems~\cite{KNS16,LZB22,Garr23}. We apply the online stochastic gradient descent to solve the problem and adopt the \emph{dynamic regret} to measure its performance, which evaluates the cumulative differences between the generated loss and the optimal loss at every time step~\cite{Zink03,BGZ15,LL19,DMP+21}. We establish a regret bound that makes explicit the dependence of the dynamic regret of online stochastic gradient descent on the cumulative distribution drifts and the gradient bias caused by the stochasticity. While a vast majority of works in online optimization literature bounds the dynamic regret in terms of the cumulative distances between optimal solutions at successive time steps, it is more natural to consider the cumulative distances between underlying distribution at successive time steps in the time-varying distribution setting. The distribution metric we adopt here is Wasserstein distance, which do away with the absolute continuity assumption on distributions at successive time steps, as needed for Kullback-Leibler (KL) divergence~\cite{CT06}. In addition, it is well-defined even when the support set is time-varying. Based on the above development, we further study a sequence of expectation loss minimization problems with a possibly nonsmooth regularizer that satisfies proximal Polyak-\L{}ojasiewicz (proximal PL) condition. We apply the online stochastic proximal gradient descent and show a regret bound that is composed of the cumulative distribution drifts and the gradient bias caused by the stochasticity.

Many applications benefit from the above framework. In particular, we apply it to the Conditional Value-at-Risk (CVaR) statistical learning problem, where the underlying distribution is time-varying. The CVaR problem focuses on making the best worst-case decision by minimizing the expected loss of the $\alpha\cdot 100\%$ worst cases, for $\alpha\in(0,1]$, which leads to a risk-averse solution. Such a solution is of particular interest in areas such as medicine, traffic and finance, when a poor solution can lead to a severe consequence. Based on the recent advances in the discovery of PL condition in the CVaR problem~\cite{Kalo22}, we establish a regret bound of online stochastic gradient descent in a CVaR problem with a time-varying underlying distribution, which, as far as we know, has barely been investigated in the literature. Specifically, we show that the assumption imposed in \cite{Kalo22} for establishing PL condition of a CVaR problem is impossible to achieve at its global optimum. Instead, we find a new non-empty subset that satisfies the PL condition while containing its global optimum. As long as the iterate lies within the subset at every time step, a regret bound of online stochastic gradient descent then follows from the said framework, which expands the repertoire of online robust optimization problems.

\subsection{Related Works}
Over the last two decades, online convex optimization has gained considerable interests in the machine learning community, for its simplicity and efficiency in dealing with large-scale data in real time. While the theory in online optimization is getting more understood, this provides a new tool in studying stochastic optimization with time-varying distribution using techniques from online optimization. For example, \cite{CZP21} studies the dynamic regret bound of online projected stochastic gradient descent when applied to a sequence of convex losses with a bounded convex feasible set. Assuming a prior knowledge on the temporal variations $\tilde{\Delta}(T)$ of the underlying distribution, the work establishes a regret bound $\mathcal{O}(\sqrt{T\tilde{\Delta}(T)})$, where $T$ is the interested time of horizon. Another example is the recent work \cite{CDH23}, which considers the error bounds of online proximal gradient descent when applied to a sequence of strongly convex loss functions, both in expectation and with high probability. The error bounds are shown to be composed of optimization error, gradient noise and time drifts. 

Beyond convexity, researchers have also explored the convergence online algorithms in solving sequences of loss functions satisfying PL condition. An earlier work \cite{ZYY+16} shows a regret bound of online multiple gradient descent with full-gradient information when solving a sequence of loss functions satisfying PL condition (or known as semi-strong convexity in the work), in which a regret bound in terms of cumulative path variations of optimal solutions is established. Recently, the work~\cite{KMD22} studies the online gradient and proximal gradient methods when the loss functions satisfy PL condition and proximal PL condition, respectively. Assuming that the gradient is contaminated by a sub-Weibull noise, the paper shows regret bounds in expectation and with high probability iteration-wise that depend on the variability of the problem and the statistics of the sub-Weibull gradient error.

A vast majority of works in online dynamic optimization capture the distribution drift via the distance between the optimal solutions of a particular loss function at consecutive time steps, which is less intuitive compared with other distribution metrics such as KL divergence and Wasserstein distance. An exception that we have noticed is the work \cite{SF20}, which shows a dynamic regret bound of online stochastic gradient descent that is composed of the cumulative Wasserstein distance between distributions at consecutive time steps when applied to a sequence of strongly convex loss functions. Yet, to the best of our knowledge, assumptions that are weaker than the strong convexity under this setting have not been studied in the literature.

\subsection{Notations}
The notation in the paper is mostly standard. We use $\|\cdot\|_1$ and $\|\cdot\|$ to denote the $\ell_1$-norm and Euclidean norm, respectively. We also use ${\rm proj}_X(\cdot)$ to denote the mapping of projection over a set $X$ and use ${\rm sgn}(\cdot)$ to denote a sign function. Moreover, we use the operator $(\cdot)_+$ to denote the operation $(\cdot)_+ = \max\{\cdot,0\}$.

\section{Online Stochastic Optimization under PL Condition}
\subsection{Problem Formulation}
Given a loss function $\LCal\colon \R^{n_x}\times \R^{n_w} \to\R$, we are interested in solving a sequence of minimization problems
\begin{equation}\label{eq:prob}
    \min_{\bm{x} \in\R^{n_x}} \left[ \FCal_t(\bm{x}) \coloneqq \E_{\bm{w}\sim \PP_t} \LCal(\bm{x},\bm{w}) \right]
\end{equation}
for $t=1,\ldots,T$ and $T$ being the horizon length. Here, $\bm{x}\in\R^{n_x}$ is a decision variable and $\bm{w}\in \R^{n_w}$ is a random parameter following an unknown distribution $\PP_t$ with probability measure $\PCal^t$ on a probability space $\Omega_t \subseteq \R^{n_w}$ at time $t$ for $t=1,\ldots,T$. Suppose that data are revealed in an online manner. Specifically, at each time step $t$, after determining a decision variable $\bm{x}_t\in\R^{n_x}$, a loss $\FCal_t(\bm{x}_t)$ is revealed. We then collect $m$ samples $\{\bm{w}_i^t\}_{i=1}^m$, which are drawn independently from the underlying distribution $\PP_t$, and use them to determine the decision variable $\bm{x}_{t+1}\in\R^{n_x}$ at the next time step. Our goal is to minimize the cumulative loss induced by decisions $\bm{x}_t$ for $t=1,\ldots,T$. This form of online stochastic optimization problem has broad applications in online learning, adaptive signal processing and online resource allocation, where decisions have to be made in real-time and the underlying distribution is unknown and time-varying.

\begin{assumption}[Lipschitzness and Differentiability of the Loss]\label{assum:shapiro}
Let $\bm{x}\in\R^{n_x}$. For $t=1,\ldots,T$, assume that the following holds:
    \begin{itemize}
    \item $\LCal(\bm{x},\cdot)$ is measurable for every $\bm{x}\in\R^{n_x}$; 
        \item $\FCal_t(\bm{x}) = \E_{\bm{w}\sim\PP_t} \LCal(\bm{x},\bm{w})$ is well-defined and finite valued;
        \item There exists a positive valued random variable $C(\bm{w})$ such that $\E_{\bm{w}\sim\PP_t}[C(\bm{w})] < \infty$, and for all $\bm{x}_1,~\bm{x}_2 \in \R^{n_x}$ in a neighborhood of $\bm{x}$ and almost every $\bm{w}\in \Omega_t$, the following inequality holds:
        \[
        |\LCal(\bm{x}_1,\bm{w}) - \LCal(\bm{x}_2,\bm{w})| \le C(\bm{w}) \|\bm{x}_1 - \bm{x}_2 \|;
        \]
        \item For almost every $\bm{w}\in \Omega_t$ the function $\LCal(\cdot,\bm{w})$ is differentiable at $\bm{x}$.
    \end{itemize}
\end{assumption}
\begin{lemma}[Differentiability {\cite[Theorem 7.44]{SDR21}}]
    Let $\bm{x}\in\R^{n_x}$. Under Assumption~\ref{assum:shapiro}, $\FCal_t(\bm{x})$ is Lipschitz continuous in a neighborhood of $\bm{x}$. Moreover, $\FCal_t(\bm{x})$ is differentiable at $\bm{x}$ and
    \[
    \nabla \FCal_t(\bm{x}) = \E_{\bm{w}\sim\PP_t} [\nabla_{\bm x} \LCal(\bm{x},\bm{w})].
    \]
\end{lemma}

Assume that $\FCal_t(\bm{x}) = \E_{\bm{w}\sim \PP_t} \LCal(\bm{x},\bm{w})$ is continuously differentiable and $\LCal(\bm{x},\bm{w})$ is differentiable with respect to (wrt) $\bm{x}$ $\PCal^t$-almost everywhere for $t=1,\ldots,T$. Here, $\LCal(\bm{x},\bm{w})$ is not necessarily differentiable everywhere, so a large class of loss functions can be included under this framework, for example, $\LCal(\bm{x},\bm{w}) = \bm{1}_{C(\bm{x})}(\bm{w})$ with some convex set $C(\bm{x})$. For $t=1,\ldots,T-1$, we update an estimate at time $t+1$ via one-step stochastic gradient descent with step size $\gamma_t>0$:
\begin{equation}\label{eq:sgd}
    \bm{x}_{t+1} =  \bm{x}_t -  \gamma_t \widehat{\nabla}\FCal_t(\bm{x}_t;\bm{w}_1^t,\ldots,\bm{w}_m^t),
\end{equation}
where $\widehat{\nabla}\FCal_t(\bm{x};\bm{w}_1^t,\ldots,\bm{w}_m^t) \approx \nabla \FCal_t(\bm{x})$ is some gradient approximation with $\E[\widehat{\nabla}\FCal_t(\bm{x};\bm{w}_1^t,\ldots,\bm{w}_m^t)] = \nabla \FCal_t(\bm{x})$. Different gradient approximations can be made in different contexts --- usually taking the average over a set of sampled gradients. However, in our setting, it is possible that given any $\bm{w} \in \Omega_t$, $\LCal(\bm{x},\bm{w})$ is non-differentiable at some $\bm{x}$, for $t=1,\ldots,T$. Hence, to make our statements precise, we introduce the following assumptions and definitions.

\begin{assumption}[Bounded Support Set]\label{assum:bdd-support}
    Every underlying distribution $\PP_t$ has a bounded support set $\Omega_t$, for $t=1,\ldots,T$.
\end{assumption}
 Under Assumption~\ref{assum:bdd-support}, we define the Clarke subdifferential of $\LCal$ wrt $\bm{x}$~\cite{LSM20}:
\[
\partial_{C,\bm{x}}  \LCal(\bm{x},\bm{w}) = \left\{\bm{s}\in \R^{n_x} \colon \bm{s}^T \bm{d} \le \limsup_{\bm{x}'\to\bm{x},t\searrow 0} \frac{\LCal(\bm{x}'+t\bm{d},\bm{w}) - \LCal(\bm{x}',\bm{w})}{t}\right\}.
\]
This set is a non-empty compact convex set~\cite[Definition (1.1)]{Clarke75}. Given $\bm{w}\in \Omega_t$, for $t=1,\ldots,T$, the Clarke subdifferential is a singleton with $\partial_{C,\bm{x}} \LCal(\bm{x},\bm{w}) = \{\nabla_{\bm x} \LCal(\bm{x},\bm{w})\}$ when $\LCal$ is differentiable at $\bm{x}$.  Having a set of samples $\{\bm{w}_i^t\}_{i=1}^m$ collected, a natural possible gradient approximation is 
\begin{equation}\label{eq:grad-approx}   \widehat{\nabla}\FCal_t(\bm{x}_t;\bm{w}_1^t,\ldots,\bm{w}_m^t) = \frac{1}{m}\sum_{i=1}^m \bm{g} (\bm{x}_t,\bm{w}_i^t)
\end{equation}
for some $\bm{g} (\bm{x}_t,\bm{w}_i^t)\in \partial_C \LCal(\bm{x}_t,\bm{w}_i^t)$. Nevertheless, there can be other possible candidates for gradient approximation, which we will see in Section~\ref{sec:cvar}. We assume any gradient approximation candidate satisfies the following assumption.
\begin{assumption}[Moments of Gradient Approximation]
For $t=1,\ldots,T$, the mean and variance of the gradient approximation $\widehat{\nabla}\FCal_t(\bm{x};\bm{w}_1^t,\ldots,\bm{w}_m^t)$ satisfies
    \begin{equation*}
    \E_{\bm{w}_1^t,\ldots,\bm{w}_m^t}[\widehat{\nabla}\FCal_t(\bm{x};\bm{w}_1^t,\ldots,\bm{w}_m^t)] = \nabla \FCal_t(\bm{x})
    \end{equation*}
    and
    \begin{equation*}
        \E_{\bm{w}_1^t,\ldots,\bm{w}_m^t}[\|\widehat{\nabla} \FCal_t (\bm{x}_t;\bm{w}_1^t,\ldots,\bm{w}_m^t) - \nabla \FCal_t (\bm{x}_t)\|^2] \le \sigma_t^2
    \end{equation*}
    for some $\sigma_t>0$.
\end{assumption}

To evaluate the performance of online SGD, we use the notion of regret:
\begin{equation}\label{eq:regret}
    {\rm Regret}(T) = \sum_{t=1}^T \E_{\{\bm{w}_1^\tau,\ldots,\bm{w}_m^\tau\}_{\tau=1}^{t-1}}[\FCal_t(\bm{x}_t) - \FCal_t^*],
\end{equation}
where $\FCal_t^* = \min_{\bm{x}} \FCal_t(\bm{x})$. Moreover, we denote $\bm{x}_t^*\in \arg\min_{\bm{x}} \FCal_t(\bm{x})$. The notion of regret is a standard performance metric in online optimization literature~\cite{LTS20}, which measures the cumulative losses deviating from the cumulative optimal losses over all time steps. Despite the fact that the vast majority of existing works derive bounds of regret via the dynamics of an optimal solution $\bm{x}_t^*$ between successive time steps~\cite{MSJ+16,LTS20,BSR18}, our goal, instead, is to bound the regret in terms of the cumulative distribution drifts and the cumulative gradient error caused by stochasticity. This bound is more intuitive since it can capture the impact of the distribution drifts on the regret. Another goal is to derive conditions that can guarantee a sublinear regret bound (i.e., ${\rm Regret}(T) \le o(T)$); in other words, the conditions that the loss $\FCal_t(\bm{x}_t)$ is getting asymptotically close to an optimal loss $\FCal_t^*$ as $\frac{1}{T}{\rm Regret}(T)\to 0$.

To characterize the distribution drifts, we employ the Wasserstein distance, which is defined below.
\begin{definition}[Wasserstein Distance]
    Let $\MCal (\R^{n_w})$ be the set of all probability distributions $\QQ$ on $\R^{n_w}$ such that $\E_{\xi \sim \QQ} \{ \|\xi\| \} < \infty$. For all $\PP, \QQ \in \MCal (\R^{n_w})$, the type-1 Wasserstein distance is defined as
\[
\Mmf(\PP, \QQ) \coloneqq \inf_{\Pi\in\mathcal{J}(\PP,\QQ)} \left\{ \int_{\R^{n_w} \times \R^{n_w}} \|\xi_1 - \xi_2 \| \Pi (d\xi_1, d \xi_2)\right\}
\]
where $\mathcal{J}(\PP,\QQ)$ is the set of joint distributions on $\xi_1$ and $\xi_2$ with marginals $\PP$ and $\QQ$, respectively.
\end{definition}

Wasserstein distance, which arises from optimal transport, has gained a lot of attention in statistics and machine learning in the last decade; see, e.g., \cite{KPT+17,LCS20}. Contrary to Kullback-Leibler divergence, Wasserstein distance is well-defined even when the support sets of two distributions are different. This provides more flexibility in the application since the support set may vary with time as well. In this work, we use the type-1 Wasserstein distance, which is also known as (aka) Kantorovich metric, to perform the analysis. The distribution drifts can then be characterized via Wasserstein distance in the following assumption.
\begin{assumption}[Bounded Distribution Drifts]\label{assum:drift}
    For $t=1,\ldots,T-1$, the probability distribution at successive time steps vary slowly by
    \[
\mathfrak{M}(\PP_{t+1}, \PP_t) \le \eta_t
    \]
    for some $\eta_t > 0$.
\end{assumption}

\subsection{Performance Analysis}
To analyze the performance of stochastic online gradient descent, we need a number of assumptions imposed on the loss function, which will be shown in Assumptions~\ref{assum:smooth}--\ref{assum:Lip-w}.

\begin{assumption}[Smoothness]\label{assum:smooth}
    Under Assumption~\ref{assum:shapiro}, for $t=1,\ldots,T$, $\FCal_t(\bm{x}) = \E_{\bm{w}\sim\PP_t} \LCal(\bm{x},\bm{w})$ is $\beta$-smooth; i.e., for any $\bm{x}, \bm{y} \in \R^{n_x}$, it holds that
\[
\| \nabla \FCal_t(\bm{y}) - \nabla \FCal_t (\bm{x}) \| \le \beta \|\bm{y} - \bm{x} \|.
\]
\end{assumption}
The smoothness property guarantees a quadratic upper approximation of the loss function at each point in the domain~\cite[Lemma 5.7]{Beck17}. This property, aka descent lemma, is a key element in proving the descent of many gradient methods.
\begin{lemma}[Descent Lemma]\label{lem:des-lem}
    Under Assumptions~\ref{assum:shapiro} and~\ref{assum:smooth}, for every $\bm{x},\bm{y}\in\R^{n_x}$ and $\bm{z}\in [\bm{x},\bm{y}] \coloneqq \{(1-\gamma) \bm{x} + \gamma \bm{y} \colon \gamma \in [0,1] \}$, we have
    \[
\FCal_t(\bm{y}) \le \FCal_t(\bm{x}) + \langle \nabla \FCal_t(\bm{z}), \bm{y} - \bm{x} \rangle + \frac{\beta}{2}\| \bm{y} - \bm{x} \|^2.
    \]
\end{lemma}
Moreover, we assume that $\FCal_t$ satisfies the Polyak-\L{}ojasiewicz condition.
\begin{assumption}[Polyak-\L{}ojasiewicz Condition]\label{assum:PL}
Under Assumption~\ref{assum:shapiro}, for $t=1,\ldots,T$, $\FCal_t(\bm{x})$ satisfies the Polyak-\L{}ojasiewicz (PL) condition on a set $\mathcal{X}$ with constant $\mu$; i.e., for all $\bm{x} \in \mathcal{X}$,
\[
\frac{1}{2}\|\nabla \FCal_t (\bm{x}) \|^2 \ge \mu (\FCal_t (\bm{x}) - \FCal_t^*).
\]
\end{assumption}
PL condition has been known to be a simple condition that guarantees a global linear convergence rate for gradient descent in offline optimization~\cite{KNS16}. Since it does not require convexity on the whole domain, it is gaining popularity especially in machine learning where loss functions are generally non-convex; see, for example, \cite{LZB22,MSM22}. Although it is not clear how to check the PL condition of $\FCal_t$ without knowing the true underlying distribution, there are scenarios that the condition reduce to the PL condition of $\LCal(\bm{x},\bm{w})$ at the mean of the parameter. For example, consider $\LCal(\bm{x},\bm{w}) = \frac{1}{2}\|g(\bm{x}) - h(\bm{w})\|^2$ for some function $g\colon\R^{n_x}\to\R^p$ and some affine function $h\colon\R^{n_w}\to\R^p$. Then, because of the fact that $\E_{\bm{w}}[\LCal(\bm{x},\bm{w})] = \nabla g(\bm{x})^T g(\bm{x}) + \langle g(\bm{x}), h(\E[\bm{w}])\rangle $, the PL condition of $\E_{\bm{w}}\LCal(\bm{x},\bm{w})$ follows if it holds for $\LCal(\bm{x},\E[\bm{w}])$.

The PL condition, combining with the smoothness property in Assumption~\ref{assum:smooth}, results in two-sided approximation bounds on the function loss from an optimal value in terms of the distance of a point from the optimal set.
\begin{lemma}[Bounds on Losses]\label{lem:QG}
    Under Assumptions~\ref{assum:shapiro},~\ref{assum:smooth} and \ref{assum:PL}, for $t=1,\ldots,T$, the following holds for all $\bm{x}\in\R^{n_x}$:
    \begin{equation}\label{eq:PL-cons}
        \frac{\mu}{2}\|\bm{x} - {\rm proj}_{\XCal^*}(\bm{x})\|^2 \le \FCal_t(\bm{x}) - \FCal_t^* \le \frac{\beta}{2} \|\bm{x} - {\rm proj}_{\XCal^*}(\bm{x})\|^2,
    \end{equation}
    where the set $\XCal^*$ is defined as the set of minimizers of $\FCal_t$.
\end{lemma}
\begin{proof}
    The first inequality follows from \cite[Thereom 2 and Appendix A]{KNS16}. The second inequality is the direct consequence of (i) descent lemma after taking expectation over $\bm{w}\sim\PP_t$, and (ii) putting $\bm{y} = \bm{x}$, $\bm{x} = {\rm proj}_{\XCal_t^*}(\bm{x})$ and $\bm{z} = {\rm proj}_{\XCal_t^*}(\bm{x})$.
\end{proof}
The next assumption that we impose on $\LCal$ is the Lipschitzness wrt the second argument.
\begin{assumption}[Lipschitzness wrt the Second Argument]\label{assum:Lip-w}
Let $\bm{x}\in\R^{n_x}$. $\LCal$ is Lipschitz continuous wrt the second argument $\bm{w}\in\R^{n_w}$; i.e., there exists a constant $K_w(\bm{x})$ depending on $\bm{x}$ such that
    \begin{equation}\label{eq:K_w}
        \left| \LCal(\bm{x},\bm{w}) - \LCal(\bm{x},\bm{w}')\right| \le K_w(\bm{x}) \cdot \|\bm{w} - \bm{w}' \|,\quad{\rm for}~i=1,\ldots,n_x.
    \end{equation}
    Moreover, we assume that there exists a universal constant $K = \max_{\bm x} K_w(\bm{x}) < \infty$ such that \eqref{eq:K_w} holds.
\end{assumption}

Similarly, we can bound the difference between two successive loss function values at the same point.
\begin{lemma}[Difference between Successive Loss Functions]\label{lem:loss-var}
Under Assumptions~\ref{assum:drift} and~\ref{assum:Lip-w}, we have
    \[
|\FCal_{t+1}(\bm{x}) - \FCal_t(\bm{x})| \le K\eta_t.
    \]
\end{lemma}
\begin{proof}
The result directly follows from
\begin{align*}
    |\FCal_{t+1}(\bm{x}) - \FCal_t(\bm{x})| = |\E_{\bm{w}\sim\PP_{t+1}}\LCal(\bm{x},\bm{w}) - \E_{\bm{w}\sim\PP_t}\LCal(\bm{x},\bm{w})|\le \Mmf(\PP_t,\PP_{t+1})\cdot K_w \le K \eta_t.
\end{align*}
\end{proof}

The last assumption on $\LCal$ is concerned with the boundedness of the expectation drift of $\nabla \LCal$ at consecutive time steps.
\begin{assumption}[Shift of Partial Derivative]\label{assum:partial-shift}
    There exists an increasing function $J\colon\R_+\to\R_+$ such that
    \begin{equation}\label{eq:J_w}
        \|  \E_{\bm{w}\sim\PP_{t+1}}[\nabla \LCal(\bm{x},\bm{w})] - \E_{\bm{w}\sim\PP_t}[\nabla \LCal(\bm{x},\bm{w})] \| \le J(\eta_t).
    \end{equation}
    for all $\bm{x}\in \R^{n_x}$, where $\eta_t$ is the bound on the distribution drift as defined in Assumption~\ref{assum:drift}.
\end{assumption}
\begin{remark}\label{rmk:partial-shift}
    Assumption~\ref{assum:partial-shift} assumes that the shifts of the expectation of every partial derivative between successive time steps are bounded by the Wasserstein distance of the two distributions. This can be satisfied when every partial derivative $\frac{\partial \LCal(\bm{x},\bm{w})}{\partial x_i}$ (for $i=1,\ldots, m$) is Lipschitz continuous. Specifically, denote $\CCal(\bm{x})$ to be the set of differentiable points of $\LCal$ wrt $\bm{w}$ and assume that $\PCal^t(\CCal(\bm{x})) = 1$ for all $t$. For any $\bm{w},\bm{w}'\in \CCal(\bm{x})$, there exists a constant $L_w(\bm{x})$ depending on $\bm{x}$ such that
        \begin{equation}\label{eq:L_w}
        \left| \frac{\partial \LCal(\bm{x},\bm{w})}{\partial x_i} - \frac{\partial \LCal(\bm{x},\bm{w}')}{\partial x_i}\right| \le L_w(\bm{x}) \cdot \|\bm{w} - \bm{w}' \|\quad {\rm for}~i=1,\ldots,n_x.
    \end{equation}
    Assume that $L \coloneqq \max_{\bm x} L_w(\bm{x}) < \infty$. Then, using Kantorovich-Rubinstein duality~\cite{thick19}, we have
    \begin{align*}
    \|  \E_{\bm{w}\sim\PP_{t+1}}[\nabla \LCal(\bm{x},\bm{w})] - \E_{\bm{w}\sim\PP_t}[\nabla \LCal(\bm{x},\bm{w})] \| &\le \sum_{i=1}^{n_x} \left|\E_{\bm{w}\sim \PP_{t+1}}\left[\frac{\partial \LCal(\bm{x},\bm{w})}{\partial x_i}\right] -  \E_{\bm{w}\sim \PP_t} \left[\frac{\partial \LCal(\bm{x},\bm{w})}{\partial x_i}\right] \right| \\
         &= \sum_{i=1}^{n_x}\inf_{\Pi(\PCal^{t+1},\PCal^t)} \left| \int_{\CCal(\bm{x})} \frac{\partial \LCal(\bm{x},\bm{w})}{\partial x_i} - \frac{\partial \LCal(\bm{x},\bm{w'})}{\partial x_i} d\Pi \right|  \\
        &\le \sum_{i=1}^{n_x} \inf_{\Pi(\PCal^{t+1},\PCal^t)}  \int_{\CCal(\bm{x})} L \|\bm{w} - \bm{w}' \| d\Pi \\
        &= n_x L\eta_t.
    \end{align*}
    However, we impose Assumption~\ref{assum:partial-shift} instead of the Lipschitzness assumption \eqref{eq:L_w} to gain some flexibility in the class of loss function $\LCal$. Under our setting, $\LCal(\bm{x},\bm{w})$ can be non-differentiable at some point $\bm{x}$. In this case, \eqref{eq:L_w} may not hold but \eqref{eq:J_w} may still hold. We will see an example in Section~\ref{sec:cvar}.
\end{remark}

Define the distance between two sets $\XCal$ and $\YCal$ by
\begin{equation*}
    {\rm dist}(\XCal,\YCal) = \inf_{\bm{x}\in\XCal,\bm{y}\in\YCal} \|\bm{x} - \bm{y}\|.
\end{equation*}
We are now ready to characterize the path variations between minimizers at successive time steps.
\begin{lemma}[Difference of Successive Optimal Values]\label{lem:sol-var}
    Under Assumptions~\ref{assum:shapiro},~\ref{assum:drift},~\ref{assum:PL}--\ref{assum:partial-shift}, the difference between the optimal loss values at successive time steps is upper bounded by
    \[
    \FCal_t^* - \FCal_{t+1}^* \le K\eta_t + \frac{1}{2\mu}J(\eta_t)^2 \quad{\rm for}~t=1,\ldots,T-1.
    \]
\end{lemma}
\begin{proof}
Applying Assumption~\ref{assum:PL}, we have
\begin{align}
   \FCal_{t+1}(\bm{x}_t^*) - \FCal_{t+1}^* \le \frac{1}{2\mu}\| \nabla \FCal_{t+1}(\bm{x}_t^*)\|^2 \le \frac{1}{2\mu}\| \nabla \FCal_{t+1}(\bm{x}_t^*) - \nabla \FCal_t(\bm{x}_t^*)\|^2. \label{eq:xx1}
\end{align} 
Moreover, applying Assumption~\ref{assum:partial-shift}, for all $\bm{x}\in\R^{n_x}$, it holds that
\begin{align}
     \| \nabla \FCal_{t+1}(\bm{x}) - \nabla \FCal_t (\bm{x})\|^2 &=
     \| \nabla \left(\E_{\bm{w}\sim\PP_{t+1}}[\LCal(\bm{x},\bm{w})]\right) - \nabla \left(\E_{\bm{w}\sim\PP_t}[\LCal(\bm{x},\bm{w})]\right) \|^2 \nonumber\\
     &=
     \|  \E_{\bm{w}\sim\PP_{t+1}}[\nabla \LCal(\bm{x},\bm{w})] - \E_{\bm{w}\sim\PP_t}[\nabla \LCal(\bm{x},\bm{w})] \|^2 \nonumber\\
     &\le J(\eta_t)^2. \label{eq:xx2}
\end{align}
Hence, using triangle inequality and the result in Lemma~\ref{lem:loss-var}, we have
\begin{align*}
    \FCal_t^* - \FCal_{t+1}^* = \FCal_t(\bm{x}_t^*) - \FCal_{t+1}(\bm{x}_t^*) + \FCal_{t+1}(\bm{x}_t^*) - \FCal_{t+1}(\bm{x}_{t+1}^*)\le K\eta_t + \frac{1}{2\mu}J(\eta_t)^2,
\end{align*}
as desired.
\end{proof}
\begin{remark}
    From the proof of Lemma~\ref{lem:sol-var}, we can also derive the distance between optimal sets at successive time steps. Specifically, let $\XCal_t^*$ be the set of minimizers of \eqref{eq:prob} at time $t$ for $t=1,\ldots,T$. Using the result of Lemma~\ref{lem:QG} and the optimality of $\bm{x}_t^*$, we have
\begin{align*}
    {\rm dist} (\XCal_t^*,\XCal_{t+1}^*)^2 
    &= \inf_{\bm{x}_t^*\in \XCal_t^*, \bm{x}_{t+1}^*\in \XCal_{t+1}^*}\| \bm{x}_t^* - \bm{x}_{t+1}^*\|^2\le \| \bm{x}_t^* - {\rm proj}_{\XCal_{t+1}^*}(\bm{x}_t^*) \|^2 \le \frac{2}{\mu} (\FCal_{t+1}(\bm{x}_t^*) - \FCal_{t+1}^*)\le \frac{J(\eta_t)^2}{\mu^2}.
\end{align*} 
\end{remark}

Armed with the above results, we are now ready to establish a regret bound of stochastic online gradient descent in distributionally time-varying online stochastic optimization.
\begin{theorem}[Regret Bound]\label{thm:PL}
    Suppose that Assumptions~\ref{assum:shapiro}--\ref{assum:partial-shift} hold and the step size satisfies $\gamma_t \equiv \gamma \in (0,\min(1/\beta,1/(2\mu)))$. Let $\zeta = -\frac{\gamma^2\beta}{2} +\gamma$, the regret can be upper bounded by
    \begin{equation}\label{eq:PLregret}
        {\rm Regret}(T) \le \frac{1}{2\mu\zeta}(\FCal_1(\bm{x}_1) - \FCal_1^*) + \frac{K}{\mu\zeta} \sum_{t=1}^{T-1} \eta_t  + \frac{1}{4\mu^2\zeta}\sum_{t=1}^{T-1} J(\eta_t)^2 + \frac{\gamma\beta}{2\mu} \sum_{t=1}^{T-1} \sigma_t^2.
    \end{equation}
\end{theorem}

\begin{proof}
    Using Lemma \ref{lem:des-lem}, we have
    \begin{align}
        \FCal_t(\bm{x}_{t+1}) - \FCal_t(\bm{x}_t) 
        &\le \langle \nabla \FCal_t (\bm{x}_t),\bm{x}_{t+1} - \bm{x}_t \rangle + \frac{\beta}{2} \|\bm{x}_{t+1} - \bm{x}_t\|^2 \nonumber\\
        & = \left\langle \nabla \FCal_t (\bm{x}_t), -\gamma\widehat{\nabla}\FCal_t (\bm{x}_t;\bm{w}_1^t,\ldots,\bm{w}_m^t) \right\rangle + \frac{\gamma^2 \beta}{2} \left\|\widehat{\nabla}\FCal_t (\bm{x}_t;\bm{w}_1^t,\ldots,\bm{w}_m^t) \right\|^2. \nonumber
    \end{align}
    
    Taking expectation of \eqref{eq:PL-smooth} wrt $\{\bm{w}_i^t\}_{i=1}^m$ given $\bm{x}_t$ yields
    \begin{align}
        &~\E_{\bm{w}_1^t,\ldots,\bm{w}_m^t \sim \PP_t}[\FCal_t(\bm{x}_{t+1}) - \FCal_t(\bm{x}_t) | \bm{x}_t] \nonumber\\
        &\le -\gamma \| \nabla \FCal_t (\bm{x}_t)\|^2 + \frac{\gamma^2 \beta}{2} \E_{\bm{w}_1^t,\ldots,\bm{w}_m^t \sim \PP_t} \left[\left\|\widehat{\nabla}\FCal_t (\bm{x}_t;\bm{w}_1^t,\ldots,\bm{w}_m^t) \right\|^2\right] \nonumber\\
        &= -\gamma\left(1 - \frac{\gamma \beta}{2}\right)\| \nabla \FCal_t (\bm{x}_t)\|^2 + \frac{\gamma^2 \beta}{2}\left(\E_{\bm{w}_1^t,\ldots,\bm{w}_m^t \sim \PP_t} \left[\left\|\widehat{\nabla}\FCal_t (\bm{x}_t;\bm{w}_1^t,\ldots,\bm{w}_m^t) \right\|^2\right] - \| \nabla \FCal_t (\bm{x}_t)\|^2\right)\nonumber\\
        &= -\gamma\left(1 - \frac{\gamma \beta}{2}\right)\| \nabla \FCal_t (\bm{x}_t)\|^2 + \frac{\gamma^2 \beta}{2}\E_{\bm{w}_1^t,\ldots,\bm{w}_m^t \sim \PP_t} \left[\left\|\widehat{\nabla}\FCal_t (\bm{x}_t;\bm{w}_1^t,\ldots,\bm{w}_m^t) - \nabla \FCal_t (\bm{x}_t)\right\|^2\right] \nonumber\\
        &\le -\gamma\left(1 - \frac{\gamma \beta}{2}\right)\| \nabla \FCal_t (\bm{x}_t)\|^2 + \frac{\gamma^2 \beta}{2}\sigma_t^2. \label{eq:desc0}
    \end{align}
    Now, writing $\zeta = -\frac{\gamma^2\beta}{2} +\gamma$, under Assumption \ref{assum:PL}, and upon applying Lemmas~\ref{lem:loss-var} and \ref{lem:sol-var}, we obtain
    \begin{align*}
        &~\E_{\bm{w}_1^t,\ldots,\bm{w}_m^t \sim \PP_t}[\FCal_{t+1}(\bm{x}_{t+1}) - \FCal_{t+1}^*|\bm{x}_t] \\
        &\le \E_{\bm{w}_1^t,\ldots,\bm{w}_m^t \sim \PP_t}[(\FCal_{t+1}(\bm{x}_{t+1}) - \FCal_t(\bm{x}_{t+1})) + (\FCal_t(\bm{x}_{t+1}) - \FCal_t(\bm{x}_t)) + (\FCal_t(\bm{x}_t) - \FCal_t^*) + (\FCal_t^* - \FCal_{t+1}^*)|\bm{x}_t] \\
        &\le 2K \eta_t + \frac{J(\eta_t)^2}{2\mu} + (1 - 2\mu\zeta)\E_{\bm{w}_1^t,\ldots,\bm{w}_m^t \sim \PP_t}[\FCal_t(\bm{x}) - \FCal_t^*|\bm{x}_t] + \frac{\gamma^2 \beta}{2} \sigma_t^2.
    \end{align*}
Since $\gamma  \in \left(0,\min\left(\frac{1}{2\mu},\frac{1}{\beta}\right)\right)$, we see that $0< 2\mu\zeta<1$. Using the above result, we can establish a regret bound:
\begin{align}
    &~\sum_{t=1}^T \E_{\{\bm{w}_1^\tau,\ldots,\bm{w}_m^\tau\}_{\tau=1}^{t-1}}[\FCal_t(\bm{x}_t) - \FCal_t^*] \nonumber\\
    &= (\FCal_1(\bm{x}_1) - \FCal_1^*)  + \sum_{t=1}^{T-1} \E_{\bm{w}_1^{t+1},\ldots,\bm{w}_m^{t+1} \sim \PP_{t+1}}[\FCal_{t+1}(\bm{x}_{t+1}) - \FCal_{t+1}^*|\bm{x}_t] \nonumber\\
    &\le (\FCal_1(\bm{x}_1) - \FCal_1^*) + 2K \sum_{t=1}^{T-1} \eta_t + \frac{1}{2\mu} \sum_{t=1}^{T-1}J(\eta_t)^2 + (1 - 2\mu\zeta)\sum_{t=1}^{T-1} \E[\FCal_t(\bm{x}_t) - \FCal_t^*] +  \frac{\gamma^2\beta}{2} \sum_{t=1}^{T-1}\sigma_t^2 . \nonumber
\end{align}

Rearranging the terms, and since $\gamma \le 1/\beta$ implies $\zeta \ge \gamma/2$,
\begin{align}
    {\rm Regret}(T) &= \sum_{t=1}^T \E_{\{\bm{w}_1^\tau,\ldots,\bm{w}_m^\tau\}_{\tau=1}^{t-1}}[\FCal_t(\bm{x}_t) - \FCal_t^*]  \nonumber\\
    &\le \frac{1}{2\mu\zeta}(\FCal_1(\bm{x}_1) - \FCal_1^*) + \frac{K}{\mu\zeta} \sum_{t=1}^{T-1} \eta_t + \frac{1}{4\mu^2\zeta}\sum_{t=1}^{T-1} J(\eta_t)^2 + \frac{\gamma^2\beta}{4\mu\zeta } \sum_{t=1}^{T-1} \sigma_t^2 \label{eq:regret-derive0}\\
    &\le \frac{1}{2\mu\zeta}(\FCal_1(\bm{x}_1) - \FCal_1^*) + \frac{K}{\mu\zeta} \sum_{t=1}^{T-1} \eta_t + \frac{1}{4\mu^2\zeta}\sum_{t=1}^{T-1} J(\eta_t)^2 + \frac{\gamma\beta}{2\mu} \sum_{t=1}^{T-1} \sigma_t^2. \label{eq:regret-derive}
\end{align}
In particular, writing $\Theta = \min(1/\beta,1/(2\mu))$ and taking $\gamma = \Theta/\sqrt{T}$, we see that
\[
\frac{1}{\zeta} = \frac{1}{-\frac{\Theta^2\beta}{2T}+ \frac{\Theta}{\sqrt{T}}}\le \frac{1}{-\frac{\Theta}{2T}+\frac{\Theta}{\sqrt{T}}} = \frac{\sqrt{T}}{-\frac{\Theta}{2\sqrt{T}}+\Theta} \le \frac{\sqrt{T}}{-\frac{\Theta}{2}+\Theta} = \frac{2\sqrt{T}}{\Theta}.
\]
Therefore, putting it back to \eqref{eq:regret-derive} yields \eqref{eq:reg-bdvar}. 
\end{proof}

As can be seen, the online stochastic gradient descent method can achieve sublinear regret when the cumulative distribution drift $\sum_{t}\eta_t$, the cumulative squared drifts of expectation of gradients $\sum_t J(\eta_t)^2$ and the cumulative variance of the gradient approximation $\sum_t \sigma_t^2$ grow sublinearly. In particular, if $J(\eta_t) \le c_0\sqrt{\eta}_t$ for some $c_0>0$ and all $t$, the condition reduces to the sublinear growth of the cumulative distribution drift $\sum_t \eta_t$ and the cumulative variance of the gradient approximation $\sum_t \sigma_t^2$. Furthermore, if the variance of the gradient approximation is constant for all $t$ (i.e., it grows linearly), online stochastic gradient descent is still able to achieve sublinear regret by picking suitable step size, as long as the cumulative distribution drift grows sufficiently slowly (such that $\sum_t\eta_t$ and $\sum_t J(\eta_t)^2$ grow no faster than $\sqrt{T}$).

\begin{remark}
    The condition on the step size $\gamma\in (0,\min(1/\beta,1/(2\mu)))$ is used to ensure the contraction of the iterate (i.e., $0<2\mu\zeta<1$) and the simplification of the regret bound in \eqref{eq:regret-derive}. A necessary and sufficient condition the step size $\gamma$ of the online stochastic gradient descent is $\gamma\in(0,\min(2/(\mu\beta), \frac{\mu-\sqrt{\mu^2 - \mu\beta}}{\mu\beta}))$, which would yield a regret bound~\eqref{eq:regret-derive0}.
\end{remark}
\begin{remark}
    As can be seen from the right-hand side of~\eqref{eq:PLregret}, the gradient error term $\sum_t\sigma_t^2$ is coupled with the step size $\gamma$. Hence, one can have control over the gradient error term using a suitable step size rule. In particular, if $\sigma_t^2 \le \sigma^2$ for some scalar $\sigma>0$ and for all $t$, setting the step size of the online SGD as $ \gamma = \min(1/\beta,1/(2\mu))/\sqrt{T}$, the regret can be upper bounded by
    \begin{equation}\label{eq:reg-bdvar}
        {\rm Regret}(T) \le M_1\sqrt{T} + M_2 \sqrt{T}\sum_{t=1}^{T-1}\eta_t + M_3 \sqrt{T}\sum_{t=1}^{T-1}J(\eta_t)^2,
    \end{equation}
    where
    \begin{align}
    M_1 = \frac{1}{\mu\Theta}(\FCal_1(\bm{x}_1) - \FCal_1^*) + \frac{\sigma^2}{2\mu}, \quad M_2 = \frac{2K}{\mu\Theta}, \quad M_3 = \frac{1}{2\mu^2\Theta},\quad \Theta = \min(1/\beta,1/(2\mu)). \label{eq:PL-smooth}
\end{align}
This fact is particularly useful when the variance of the gradient error does not diminish over time.
\end{remark}

\begin{remark}\label{rmk:stepsize}
    For simplicity, we keep the step size $\gamma_t$ constant throughout all time steps $t$. In particular, if the variance of the measurement noise is constant at all time steps, Theorem~\ref{thm:PL} states that one may need to set the time horizon of $T$ in advance for the selection of the suitable step size $\gamma = \min(1/\beta,1/(2\mu))/\sqrt{T}$ of online stochastic gradient descent. However, in fact, the proof still follows if the step size is chosen to be $\gamma_t = \min(1/\beta,1/(2\mu))/\sqrt{t}$ for $t=1,\ldots,T$. Specifically, similar regret bound can be achieved by considering $\zeta_t = -\frac{\gamma_t^2\beta}{2} +\gamma_t$ and using the fact that
    \[
    \frac{1}{\zeta_t} = \frac{1}{-\frac{\gamma_t^2\beta}{2}+\gamma_t} = \frac{2}{-\gamma_t^2\beta+2\gamma_t} = \frac{2}{\gamma_t}\cdot\frac{1}{2 - \gamma_t\beta} \le \frac{2}{\gamma_t}\cdot\frac{1}{2 - \frac{1}{\beta}\cdot\beta} = \frac{2}{\gamma_t} \le \frac{2}{\gamma_T}
    \]
    for all $t$. In Section~\ref{sec:sim}, we will see that the latter step size would yield a better performance of online stochastic gradient descent. Moreover, such a step size does not rely on the information of time of horizon, which may be more useful in practice.
\end{remark}

    
\section{Online Stochastic Optimization under Proximal PL Condition}
In the previous section, we consider the minimization problem of a smooth data fidelity loss function where the underlying distribution the data is time-varying. Yet, its regularized version is also of interest since one may want to impose some structure on the decision vector. In this section, we show that similar regret bound can be developed for stochastic online proximal gradient descent given a sequence of loss functions satisfying the proximal PL condition.
\subsection{Problem Formulation}
Let $\FCal_t(\bm{x})\coloneqq \E_{\bm{w}\sim\PP_t} \LCal(\bm{x},\bm{w})$. In this section, we consider a sequence of optimization problems
\begin{equation}\label{eq:prox-prob}
    \min_{\bm{x} \in\R^{n_x}} [\GCal_t(\bm{x}) \coloneqq \FCal_t(\bm{x}) + \RCal(\bm{x})]
\end{equation}
for $t=1,\ldots,T$ for some potentially non-smooth convex regularizer $\RCal\colon\R^{n_x}\to\R$. The regularizer $\RCal$ can be used to impose structures on the decision vector, for example, $\RCal(\bm{x}) = \|\bm{x}\|_1$ imposes sparsity on the decision vector. 

Under Assumptions~\ref{assum:shapiro} and \ref{assum:bdd-support}, for $t=1,\ldots,T-1$, we employ the one-step stochastic proximal gradient descent
\begin{align}\label{eq:forw-grad}
    \bm{x}_{t+1} &= {\rm prox}_{\gamma_t \RCal}(\bm{x}_t - \gamma_t\widehat{\nabla}\FCal_t(\bm{x}_t;\bm{w}_1^t,\ldots,\bm{w}_m^t)) \nonumber\\
    &= \arg\min_{\bm y} \left\{\left\langle \widehat{\nabla}\FCal_t(\bm{x}_t;\bm{w}_1^t,\ldots,\bm{w}_m^t), \bm{y} - \bm{x}_t\right\rangle + \frac{1}{2\gamma_t} \|\bm{y} - \bm{x}_t\|^2 + \RCal(\bm{y}) - \RCal(\bm{x}_t)\right\}.
\end{align}
where $\widehat{\nabla}\FCal_t$ is defined in \eqref{eq:grad-approx}. We, again, use the notion of regret to evaluate its performance, namely,
\[
    {\rm Regret}(T) = \sum_{t=1}^T \E_{\{\bm{w}_1^\tau,\ldots,\bm{w}_m^\tau\}_{\tau=1}^{t-1}}[\GCal_t(\bm{x}_t) - \GCal_t^*],
\]
where $\GCal_t^*=\min_{\bm{x}} \GCal_t(\bm{x}_t)$ is the minimum loss. We also denote $\bm{x}_t^*\in\arg\min_{\bm{x}} \GCal_t(\bm{x}_t)$.
Suppose that Assumptions~\ref{assum:shapiro}--\ref{assum:smooth} and \ref{assum:Lip-w}--\ref{assum:partial-shift} hold. Moreover, we assume that the proximal Polyak-\L{}ojasiewicz condition holds for $\GCal_t$ for all $t=1,\ldots,T$.
\begin{assumption}[Proximal Polyak-\L{}ojasiewicz Condition]\label{assum:proxPL}
Under Assumption~\ref{assum:shapiro}, for $t=1,\ldots,T$, $\GCal_t$ satisfies the proximal Polyak-\L{}ojasiewicz (proximal PL) condition on a set $\mathcal{X}$ with constant $\mu$; i.e., for all $\bm{x} \in \mathcal{X}$,
\[
\frac{1}{2}\DCal_\RCal^t (\bm{x},\beta) \ge \mu (\GCal_t (\bm{x}) - \GCal_t^*),
\]
where
\begin{equation*}
    \DCal_\RCal^t (\bm{x},\delta) := -2\delta \min_{\bm y} \left\{ \langle \nabla \FCal_t(\bm{x}), \bm{y} - \bm{x}\rangle  + \frac{\delta}{2}\|\bm{y} - \bm{x}\|^2 + \RCal(\bm{y}) - \RCal(\bm{x})\right\}.
\end{equation*}
\end{assumption}
Proximal PL condition is a generalization of PL condition in non-smooth optimization. It is known that problems like support vector machine and $\ell_1$ regularized least squares satisfy proximal PL condition; see more examples in \cite[Section 4.1 and Appendix G]{KNS16}. Similar to PL condition, the quadratic growth property is also implied for functions satisfying proximal PL condition.
\begin{lemma}[Quadratic Growth]\label{lem:prox-qg}
    Let $\FCal_t\colon\R^{n_x}\to\R$ be a function that satisfies proximal PL condition. Then, under Assumption~\ref{assum:smooth}, there exists a constant $\xi>0$ such that for every $\bm{x}\in\R^{n_x}$, the following holds
    \begin{equation}\label{eq:prox-qg}
        \frac{\xi}{2}\|\bm{x} - {\rm proj}_{\XCal_t^*}(\bm{x})\|^2 \le \GCal_t(\bm{x}) - \GCal_t^*.
    \end{equation}
\end{lemma}
This is a direct consequence of the equivalence of proximal PL condition, proximal error bound condition and quadratic growth \cite[Appendix G]{KNS16},\cite[Corollary 3.6]{DL18}. Having the proximal PL condition, we can also bound the distance between two successive optimal sets and the difference between two successive loss function values at the same point.

\begin{lemma}[Difference between Successive Loss Functions]\label{lem:prox-loss-var}
Under Assumptions~~\ref{assum:drift} and~\ref{assum:Lip-w}, we have
    \[
|\GCal_{t+1}(\bm{x}) - \GCal_t(\bm{x})| \le K\eta_t.
    \]
\end{lemma}
The lemma directly follows from Lemma~\ref{lem:loss-var}. Collecting all the results, we can now establish a regret bound of stochastic online proximal gradient descent.

\begin{lemma}[Difference of Successive Optimal Values]\label{lem:prox-sol-var}
    For $t=1,\ldots,T$. Under Assumptions~\ref{assum:shapiro},~\ref{assum:drift}, ~\ref{assum:Lip-w}--\ref{assum:proxPL}, we have
\begin{equation*}
    \GCal_t^* - \GCal_{t+1}^* \le K\eta_t + \frac{J(\eta_t)^2}{2\mu}.
\end{equation*}
\end{lemma}
\begin{proof}
    
Note that
\begin{align}
\GCal_{t+1}(\bm{x}_t^*) - \GCal_{t+1}^* &\le \frac{1}{2\mu}\DCal_\RCal^{t+1} (\bm{x}_t^*,\beta) \nonumber\\
    &= -\frac{\beta}{\mu}\cdot \min_{\bm y} \left\{ \langle \nabla \FCal_{t+1}(\bm{x}_t^*),\bm{y} - \bm{x}_t^* \rangle + \frac{\beta}{2}\|\bm{y} - \bm{x}_t^*\|^2 + \RCal(\bm{y}) - \RCal(\bm{x}_t^*)\right\} \nonumber\\
    &\le -\frac{\beta}{\mu} \Bigg( \min_{\bm y} \left\{ \langle \nabla \FCal_{t+1}(\bm{x}_t^*) - \nabla \FCal_t(\bm{x}_t^*),\bm{y} - \bm{x}_t^*\rangle  + \frac{\beta}{2}\|\bm{y} - \bm{x}_t^*\|^2\right\} \nonumber\\
    &\quad+\min_{\bm y} \left\{ \langle \nabla\FCal_t(\bm{x}_t^*),\bm{y} - \bm{x}_t^*\rangle + \RCal(\bm{y}) - \RCal(\bm{x}_t^*) \right\} \Bigg) \nonumber\\
    &\le -\frac{\beta}{\mu}\cdot \min_{\bm y} \left\{ \langle \nabla \FCal_{t+1}(\bm{x}_t^*) - \nabla \FCal_t(\bm{x}_t^*),\bm{y} - \bm{x}_t^*\rangle  + \frac{\beta}{2}\|\bm{y} - \bm{x}_t^*\|^2 \right\}. \label{eq:yy2}
\end{align}
The last inequality follows from the optimality of $\bm{x}_t^*$. Also, we can easily obtain the global optimum $\bm{y} = \bm{x}_t^* + \frac{1}{\beta}(\nabla \FCal_t(\bm{x}_t^*) - \nabla \FCal_{t+1}(\bm{x}_t^*))$ for the minimization problem in the last inequality. Plugging this back into \eqref{eq:yy2} and using the argument in \eqref{eq:xx2}, for all $\bm{x}\in\R^{n_x}$, we obtain
\begin{align*}
\GCal_{t+1}(\bm{x}_t^*) - \GCal_{t+1}^* &\le \frac{1}{2\mu}\|\nabla \FCal_{t+1}(\bm{x}_t^*) - \nabla \FCal_t(\bm{x}_t^*)\|^2 \\
      &=\frac{1}{2\mu}
     \| \nabla \left(\E_{\bm{w}\sim\PP_{t+1}}[\LCal(\bm{x},\bm{w})]\right) - \nabla \left(\E_{\bm{w}\sim\PP_t}[\LCal(\bm{x},\bm{w})]\right) \|^2 \\
     &= \frac{1}{2\mu}\|  \E_{\bm{w}\sim\PP_{t+1}}\left[\nabla_{\bm x} \LCal(\bm{x},\bm{w}) \right] - \E_{\bm{w}\sim\PP_t} \left[\nabla_{\bm x} \LCal(\bm{x},\bm{w}) \right] \|^2 \\
     &\le \frac{J(\eta_t)^2}{2\mu}.
\end{align*}
Hence, using triangle inequality and result in Lemma~\ref{lem:prox-loss-var}, we have
\begin{equation*}
    \GCal_t^* - \GCal_{t+1}^* = \GCal_t(\bm{x}_t^*) - \GCal_{t+1}(\bm{x}_t^*) + \GCal_{t+1}(\bm{x}_t^*) - \GCal_{t+1}(\bm{x}_{t+1}^*) \le K\eta_t + \frac{J(\eta_t)^2}{2\mu},
\end{equation*}
as desired.
\end{proof}

\begin{remark}
    From the proof of Lemma~\ref{lem:prox-sol-var}, we can also derive the distance between optimal sets at successive time steps. Specifically, let $\XCal_t^*$ be the set of minimizers of \eqref{eq:prob} at time $t$ for $t=1,\ldots,T$. Using the result of Lemma~\ref{lem:prox-qg} and the optimality of $\bm{x}_t^*$, we have
    \begin{align*}
    {\rm dist} (\XCal_t^*,\XCal_{t+1}^*)^2 
    &= \inf_{\bm{x}_t^*\in \XCal_t^*, \bm{x}_{t+1}^*\in \XCal_{t+1}^*}\| \bm{x}_t^* - \bm{x}_{t+1}^*\|^2 \le \| \bm{x}_t^* - {\rm proj}_{\XCal_{t+1}^*}(\bm{x}_t^*) \|^2 \le \frac{2}{\xi} (\GCal_{t+1}(\bm{x}_t^*) - \GCal_{t+1}^*)\le \frac{J(\eta_t)^2}{\xi\mu}.
\end{align*}
\end{remark}

Having the above set up, we can establish a regret bound of online stochastic proximal gradient descent similar to Theorem~\ref{thm:PL}. 
\begin{theorem}\label{thm:prox-PL}
Suppose that Assumptions~\ref{assum:shapiro}--\ref{assum:smooth},~\ref{assum:Lip-w}--\ref{assum:proxPL} hold. For any step size $\gamma_t\equiv \gamma \in (0,1/\beta)$, the regret can be upper bounded by
\begin{equation}\label{eq:proxPLregret}
    {\rm Regret}(T) \le \frac{1}{2\mu\gamma}(\GCal_1(\bm{x}_1) - \GCal_1^*) + \frac{K}{\mu\gamma}\sum_{t=1}^{T-1} \eta_t + \frac{1}{4\mu^2\gamma}\sum_{t=1}^{T-1} J(\eta_t)^2  +\frac{1}{4\mu}\sum_{t=1}^{T-1}\sigma_t^2.
\end{equation}
\end{theorem}

\begin{proof}
Applying Assumption~\ref{assum:smooth} and using the result in Lemma~\ref{lem:des-lem}, we can write
    \begin{align}
        \GCal_t(\bm{x}_{t+1}) - \GCal_t(\bm{x}_t) &= \FCal_t(\bm{x}_{t+1}) - \FCal_t(\bm{x}_t) + \RCal(\bm{x}_{t+1}) - \RCal(\bm{x}_t) \nonumber\\
        &\le \langle \nabla \FCal_t(\bm{x}_t),\bm{x}_{t+1} -\bm{x}_t \rangle +\frac{\beta}{2}\|\bm{x}_{t+1} -\bm{x}_t\|^2 + \RCal(\bm{x}_{t+1}) - \RCal(\bm{x}_t). \label{eq:pfg0}
    \end{align}
        Since the update $\bm{x}_{t+1}$ is determined by the sampling data $\{\bm{w}_i^t\}_{i=1}^m$ and the previous update $\bm{x}_t$, taking expectation over \eqref{eq:pfg0} yields
    \begin{align}
&~\E_{\bm{w}_1^t,\ldots,\bm{w}_m^t}\left[\GCal_t(\bm{x}_{t+1}) - \GCal_t(\bm{x}_t)|\bm{x}_t\right] \nonumber\\
        &\le \underbrace{\left[ \langle \nabla \FCal_t(\bm{x}_t), \bm{x}_{t+1}' - \bm{x}_t \rangle +\frac{1}{2\gamma}\|\bm{x}_{t+1}' -\bm{x}_t\|^2 + \RCal(\bm{x}_{t+1}') - \RCal(\bm{x}_t)\right]}_{\rm (I)} \nonumber\\
        &+ \underbrace{\E_{\bm{w}_1^t,\ldots,\bm{w}_m^t}\Bigg[\langle \nabla \FCal_t(\bm{x}_t), \bm{x}_{t+1} - \bm{x}_{t+1}' \rangle +\frac{\beta}{2}\|\bm{x}_{t+1} -\bm{x}_t\|^2- \frac{1}{2\gamma}\|\bm{x}_{t+1}' -\bm{x}_t\|^2  + \RCal(\bm{x}_{t+1}) - \RCal(\bm{x}_{t+1}')\Bigg]}_{\rm (II)}, \label{eq:pfg1}
    \end{align}
    where
    \begin{equation*}
        \bm{x}_{t+1}' = \arg\min_{\bm z} \left\{ \langle \nabla \FCal_t(\bm{x}_t), \bm{z} - \bm{x}_t\rangle +\frac{1}{2\gamma} \|\bm{z} - \bm{x}_t \|^2 + \RCal(\bm{z}) - \RCal(\bm{x}_t)\right\}.
    \end{equation*}
    Under Assumption~\ref{assum:proxPL}, we can bound (I)
    \begin{equation}\label{eq:proxPL}
        \langle \nabla \FCal_t(\bm{x}_t), \bm{x}_{t+1}' - \bm{x}_t \rangle +\frac{1}{2\gamma}\|\bm{x}_{t+1}' -\bm{x}_t\|^2 + \RCal(\bm{x}_{t+1}') - \RCal(\bm{x}_t) = -\gamma\DCal_\RCal^t\left(\bm{x}_t,\frac{1}{\gamma}\right) \le -2\mu\gamma (\GCal_t(\bm{x}_t)-\GCal_t^*).
    \end{equation}
    Next, (II) can be written as
    \begin{align}
        & \E_{\bm{w}_1^t,\ldots,\bm{w}_m^t}\Bigg[ \left\langle \nabla \FCal_t(\bm{x}_t) - \widehat{\nabla}\FCal_t(\bm{x}_t;\bm{w}_1^t,\ldots,\bm{w}_m^t), \bm{x}_{t+1} - \bm{x}_{t+1}' \right\rangle\Bigg]  \nonumber\\
        & +\E_{\bm{w}_1^t,\ldots,\bm{w}_m^t}\Bigg[ \left\langle \widehat{\nabla}\FCal_t(\bm{x}_t;\bm{w}_1^t,\ldots,\bm{w}_m^t), \bm{x}_{t+1} - \bm{x}_{t+1}' \right\rangle+\frac{\beta}{2}\|\bm{x}_{t+1} -\bm{x}_t\|^2 + \RCal(\bm{x}_{t+1}) - \RCal(\bm{x}_{t+1}')- \frac{1}{2\gamma}\|\bm{x}_{t+1}' -\bm{x}_t\|^2 \Bigg]. \label{eq:pg2}
    \end{align}
    Recalling that the updating rule is given by
    \begin{align}
        \bm{x}_{t+1} = \arg\min_{\bm z} \left\{\left\langle \widehat{\nabla}\FCal_t(\bm{x}_t;\bm{w}_1^t,\ldots,\bm{w}_m^t), \bm{z} - \bm{x}_t \right\rangle + \frac{1}{2\gamma}\|\bm{z} - \bm{x}_t\|^2 + \RCal(\bm{z}) - \RCal(\bm{x}_t) \eqqcolon H(\bm{z})\right\}.\label{eq:pfg_up}
    \end{align}
    Given $\{\bm{w}_i^t\}_{i=1}^m$ and the assumption that $\RCal$ is convex, $H$ is strongly convex. Therefore, by the optimality of $\bm{x}_{t+1}$, we have
    \[
    H(\bm{x}_{t+1}') \ge H(\bm{x}_{t+1}) + \frac{1}{2\gamma} \|\bm{x}_{t+1}' - \bm{x}_{t+1} \|^2.
    \]
    That is,
    \begin{align}
        &\left\langle \widehat{\nabla}\FCal_t(\bm{x}_t;\bm{w}_1^t,\ldots,\bm{w}_m^t), \bm{x}_{t+1} - \bm{x}_{t+1}' \right\rangle + \frac{1}{2\gamma}\|\bm{x}_{t+1} -\bm{x}_t\|^2 \nonumber\\
        &+ \frac{1}{2\gamma}\|\bm{x}_{t+1}' - \bm{x}_{t+1}\|^2 + \RCal(\bm{x}_{t+1}) - \RCal(\bm{x}_{t+1}') - \frac{1}{2\gamma}\|\bm{x}_{t+1}' -\bm{x}_t\|^2 \le 0. \label{eq:pg_strcvx}
    \end{align}
    Therefore,
    \begin{align}
        &\E_{\bm{w}_1^t,\ldots,\bm{w}_m^t}\Bigg[ \left\langle \widehat{\nabla}\FCal_t(\bm{x}_t;\bm{w}_1^t,\ldots,\bm{w}_m^t), \bm{x}_{t+1} - \bm{x}_{t+1}' \right\rangle+\frac{\beta}{2}\|\bm{x}_{t+1} - \bm{x}_t\|^2 + \RCal(\bm{x}_{t+1}) - \RCal(\bm{x}_{t+1}') - \frac{1}{2\gamma}\|\bm{x}_{t+1}' -\bm{x}_t\|^2 \Bigg] \nonumber\\
        &\le \E_{\bm{w}_1^t,\ldots,\bm{w}_m^t}\Bigg[\frac{1}{2}\left(\beta - \frac{1}{\gamma}\right)\|\bm{x}_{t+1} -\bm{x}_t\|^2 - \frac{1}{2\gamma}\|\bm{x}_{t+1}' - \bm{x}_{t+1}\|^2 \Bigg]. \label{eq:pg3}
    \end{align}
    Putting \eqref{eq:pg3} back to \eqref{eq:pg2} and using Young's inequality~\cite[Proposition 2.7]{ABMY14}, (II) can be bounded by
    \begin{align}
        &~\E_{\bm{w}_1^t,\ldots,\bm{w}_m^t}\Bigg[ \left\langle \nabla \FCal_t(\bm{x}_t) - \widehat{\nabla}\FCal_t(\bm{x}_t;\bm{w}_1^t,\ldots,\bm{w}_m^t), \bm{x}_{t+1} - \bm{x}_{t+1}' \right\rangle + \frac{1}{2}\left(\beta - \frac{1}{\gamma}\right)\|\bm{x}_{t+1} -\bm{x}_t\|^2 - \frac{1}{2\gamma}\|\bm{x}_{t+1}' - \bm{x}_{t+1}\|^2\Bigg] \nonumber\\
        &\le \frac{\gamma}{2}\E_{\bm{w}_1^t,\ldots,\bm{w}_m^t}\Bigg[ \left\| \nabla \FCal_t(\bm{x}_t) - \widehat{\nabla}\FCal_t(\bm{x}_t;\bm{w}_1^t,\ldots,\bm{w}_m^t)\right\|^2\Bigg] + \E_{\bm{w}_1^t,\ldots,\bm{w}_m^t}\Bigg[ \frac{1}{2}\left(\beta - \frac{1}{\gamma}\right)\|\bm{x}_{t+1} -\bm{x}_t\|^2\Bigg] \nonumber\\
        &\le \frac{\gamma}{2}\sigma_t^2 + \E_{\bm{w}_1^t,\ldots,\bm{w}_m^t}\Bigg[ \frac{1}{2}\left(\beta - \frac{1}{\gamma}\right)\|\bm{x}_{t+1} -\bm{x}_t\|^2\Bigg].\label{eq:pg4} 
    \end{align}
Since $\gamma \le 1/ \beta$, putting \eqref{eq:proxPL} and \eqref{eq:pg4} into \eqref{eq:pfg1}, we have
\begin{align*}
&~\E_{\bm{w}_1^t,\ldots,\bm{w}_m^t}\left[\GCal_t(\bm{x}_{t+1}) - \GCal_t(\bm{x}_t)|\bm{x}_t\right] \le -2\mu\gamma \E_{\bm{w}_1^t,\ldots,\bm{w}_m^t}[\GCal_t(\bm{x}_t)-\GCal_t^*|\bm{x}_t] + \frac{\gamma}{2}\sigma_t^2.
    \end{align*}
Therefore, given $\bm{x}_t\in\R^{n_x}$,
    \begin{align}
        &~\E_{\bm{w}_1^t,\ldots,\bm{w}_m^t\sim\PP_t}[\GCal_{t+1}(\bm{x}_{t+1}) - \GCal_{t+1}^*|\bm{x}_t] \nonumber\\
        &=\E[(\GCal_{t+1}(\bm{x}_{t+1}) - \GCal_t(\bm{x}_{t+1})) + (\GCal_t(\bm{x}_{t+1}) - \GCal_t(\bm{x}_t)) + (\GCal_t(\bm{x}_t^*) - \GCal_t^*) + (\GCal_t^* - \GCal_{t+1}^*)] \nonumber\\
        &\le 2K\eta_t + \frac{J(\eta_t)^2}{2\mu} + (1 - 2\mu\gamma)(\GCal_t(\bm{x}_t) - \GCal_t^*) + \frac{\gamma}{2}\sigma_t^2. \nonumber
    \end{align}
Summing the terms up,
    \begin{align*}
        &~\sum_{t=1}^T \E_{\{\bm{w}_1^\tau,\ldots,\bm{w}_m^\tau\}_{\tau=1}^{t-1}}[\GCal_t(\bm{x}_t) - \GCal_t^*] \nonumber\\
        &= (\GCal_1(\bm{x}_1) - \GCal_1^*) + \sum_{t=1}^{T-1} \E_{\bm{w}_1^t,\ldots,\bm{w}_m^t\sim\PP_t}[\GCal_{t+1}(\bm{x}_{t+1}) - \GCal_{t+1}^*] \nonumber\\
        &\le (\GCal_1(\bm{x}_1) - \GCal_1^*) + 2K\sum_{t=1}^{T-1}\eta_t + \frac{1}{2\mu}\sum_{t=1}^{T-1} J(\eta_t)^2 + \sum_{t=1}^{T-1} \left(1-2\mu\gamma\right) \E[\GCal_t(\bm{x}_t) - \GCal_t^*]+\frac{\gamma}{2}\sum_{t=1}^{T-1}\sigma_t^2.
    \end{align*}
Rearranging the terms, the regret is upper bounded by
    \begin{align*}
        {\rm Regret}(T) &= \sum_{t=1}^T \E_{\bm{w}_1^t,\ldots,\bm{w}_m^t\sim\PP_t} \left[\GCal_t(\bm{x}_t) - \GCal_t^* \right] \nonumber\\
        &\le \frac{1}{2\mu\gamma}(\GCal_1(\bm{x}_1) - \GCal_1^*) + \frac{K}{\mu\gamma}\sum_{t=1}^{T-1} \eta_t + \frac{1}{4\mu^2\gamma}\sum_{t=1}^{T-1} J(\eta_t)^2  +\frac{1}{4\mu}\sum_{t=1}^{T-1}\sigma_t^2.
    \end{align*}

Theorem~\ref{thm:prox-PL} shows that the online stochastic proximal gradient descent method can achieve sublinear regret when the cumulative distribution drift $\sum_{t}\eta_t$, the cumulative squared drifts of expectation of gradients $\sum_t J(\eta_t)^2$ and the cumulative variance of the gradient approximation $\sum_t \sigma_t^2$ grow sublinearly. However, if the variance of the gradient approximation is constant throughout all $t$, sublinear regret bounds can no longer be achieved. This is due to the technical challenge caused by the nonsmoothness of the regularizer. Yet, in Section~\ref{sec:sim}, we will see numerical examples that a sublinear regret of online stochastic proximal gradient descent can be observed while the cumulative variance of gradient approximation grows linearly given a suitable step size.
\begin{remark}\label{rmk:prox-ss}
    Unlike Theorem~\ref{thm:PL}, the gradient error term $\sum_t\sigma_t$ shown in the right-hand side of the regret bound~\eqref{eq:proxPLregret} does not couple with any step size, implying that we cannot control the term using a suitable step size rule. In other words, if the gradient error does not diminish, Theorem~\ref{thm:prox-PL} cannot guarantee a sublinear regret bound of online stochastic proximal gradient descent. However, sublinear regret can still be observed empirically using a suitable step size rule; see Section~\ref{sec:sim}. This suggests that it is possible to achieve a tighter regret bound of online stochastic proximal gradient descent given some assumptions on the regularizer. We will leave this as a future work.
\end{remark}
\end{proof}

\section{Application to CVaR Statistical Learning}\label{sec:cvar}
Without assuming convexity, this framework can be applied to a broader class of loss functions. In this section, we show how time-varying CVaR learning problem benefits from the above setup. In the following, the notation might be slightly different from the above sections, which we will define in due course.

\subsection{CVaR Formulation and Preliminaries}\label{sec:cvar-prelim}
Consider a known parametric family of functions $\FCal \coloneqq \{\phi \colon \R^n \to \R|\phi(\cdot)\equiv f(\cdot,\bm{\theta}),\bm{\theta}\in\R^n \}$, called a hypothesis class. At each time $t=1,\ldots,T$, we collect samples $(\bm{x},y)\in\R^d\times\R$ from an unknown distribution $\PP_t$ on example space $\Omega_t$ and would like to find $\bm{\theta}_t^*\in\R^n$ that can best describe the relation between input $\bm{x}$ and output $y$. Specifically, we use a loss function $\ell\colon \R\times\R \to \R$ to measure the discrepancy between the quantity of an admissible predictor $f(\bm{x},\bm{\theta})$ and the output $y$ for each sample $(\bm{x},y)$, and minimize an expected loss
\begin{equation}\label{eq:loss}
    \inf_{\bm{\theta}\in \R^n} \E_{(\bm{x},y)\sim \PP_t}\{\ell(f(\bm{x},\bm{\theta}),y)\} .
\end{equation}
at each time step $t$. A fundamental issue about this formulation is that it is risk-neutral. In some applications, for example, making medical decisions and portfolio management, one of the objectives is to avoid worst-case scenarios, and therefore, a robust risk measure is of more interest. In view of this, one of the most popular risk measures in theory and practice is CVaR, which is defined as
\begin{equation*}
    {\rm CVaR}^\alpha(Z) \coloneqq \inf_{h \in \R} \left\{h + \frac{1}{\alpha}\E\{(Z - h)_+\}\right\}
\end{equation*}
at confidence level $\alpha \in (0,1]$ for an integrable random loss $Z$. Putting $Z= \ell(f(\bm{x},\bm{\theta}),y)$, we can reformulate problem \eqref{eq:loss} using CVaR measure over variables $(\bm{\theta},h)$ as
\begin{equation*}
    \inf_{(\bm{\theta},h)\in \R^n\times\R}  \E_{(\bm{x},y)\sim \PP_t}\left\{ h + \frac{1}{\alpha}(\ell(f(\bm{x},\bm{\theta}),y) - h)_+ \right\}.
\end{equation*}

Intuitively, ${\rm CVaR}^\alpha(Z)$ is the mean of the worst $\alpha \cdot 100\%$ of the values of $Z$. To see this, we define the Value-at-Risk (VaR) of $Z$ at level $\alpha\in(0,1]$, which is given by
\begin{equation*}
    {\rm VaR}^\alpha(Z) \coloneqq \inf \{z\in\R \colon \PCal^t(\{Z \le z\}) \ge 1 - \alpha\};
\end{equation*}
in other words, the VaR can be understood as the left-side $(1-\alpha)$-quantile of the distribution of $Z$~\cite{CBS+21}. The results in~\cite[Theorem 6.2]{SDR21} show that the CVaR of $Z$ at level $\alpha\in(0,1]$ is equivalent to an expectation conditioned on random variables greater than VaR; i.e.,
\begin{equation}
    {\rm CVaR}^\alpha(Z) = \E(Z|Z\ge {\rm VaR}^\alpha(Z)).
\end{equation}
Since $\PCal^t(Z > {\rm VaR}^\alpha(Z)) = \alpha$, one can deduce that $\PCal^t(Z > {\rm CVaR}^\alpha(Z)) < \alpha$.

\subsection{CVaR with Time-Varying Distribution}
Let $\alpha\in(0,1]$. Denote $\ell_\alpha\colon\R^n \times \R\times\Omega \to \R$ by
\begin{equation}\label{eq:ell-alpha}
    \ell_\alpha(\bm{\theta},h;\bm{x},y) \coloneqq h + \frac{1}{\alpha}(\ell(f(\bm{x},\bm{\theta}),y) - h)_+.
\end{equation}
Then, for $t=1,\ldots,T$, our goal is to solve
\begin{equation}\label{eq:cvar-prob}
\min_{(\bm{\theta},h)} L_\alpha^t(\bm{\theta},h)
\end{equation}
where $L_\alpha^t \colon \R^n\times \R \to \R$ is given by
\begin{equation}\label{eq:cvar-loss}
    L_\alpha^t(\bm{\theta},h) \coloneqq \E_{(\bm{x},y)\sim\PP_t}[\ell_\alpha(\bm{\theta},h;\bm{x},y)] = \E_{(\bm{x},y)\sim\PP_t}\left\{ h + \frac{1}{\alpha}(\ell(f(\bm{x},\bm{\theta}),y) - h)_+\right\}.
\end{equation}
For the sake of notational simplicity, we assume that every distribution $\PP_t$ shares the same support set $\Omega_t\equiv \Omega$ for $t=1,\ldots,T$.

\begin{assumption}\label{assum:exp-diff}
The following statements hold:
    \begin{itemize}
        \item[(a)] For each $\bm{\theta}\in \R^n$,$\ell(f(\bm{x},\cdot),y)$ is $C_{\bm \theta}(\bm{x},y)$-Lipschitz on a neighborhood $\bm{\theta}$ for $\PCal^t$-almost all $(\bm{x},y)$, where $\E_{\PCal^t}\{C_{\bm \theta}(\bm{x},y)\}<\infty$.
        \item[(b)] $\ell(f(\bm{x},\cdot),\cdot)$ is differentiable at $\bm{\theta}$ for $\PCal^t$-almost all $(\bm{x},y)$, and $\PCal^t(\ell(f(\bm{x},\bm{\theta}),y)=h)\equiv 0$ for all $(\bm{\theta},h)\in \R^n \times \R$.
    \end{itemize}
\end{assumption}
Under Assumption~\ref{assum:exp-diff}, differentiation may be interchanged with expectation for $L_\alpha^t$ \cite[Section 7.2.4]{SDR21}. Moreover, the function $L_\alpha^t$ is differentiable~\cite[Lemma 1]{BBB17} and the gradient representation for every $(\bm{\theta},h)\in\R^n \times \R$ is given by
\begin{align}\label{eq:cvar-grad}
    \nabla L_\alpha^t (\bm{\theta}, h) = 
    \begin{bmatrix}
    \frac{1}{\alpha}\E_{(\bm{x},y)\sim\PP_t}\{\bm{1}_{\ACal(\bm{\theta},h)}(\bm{x},y)\nabla_{\bm \theta} \ell(f(\bm{x},\bm{\theta}),y) \}\\
    -\frac{1}{\alpha}\PCal^t(\ACal(\bm{\theta},h))) + 1
    \end{bmatrix},
\end{align}
where the event-valued multifunction $\ACal \colon \R^n \times \R \rightrightarrows \Omega $ is defined as
\begin{equation}\label{eq:cvar-event}
    \ACal(\bm{\theta},h) \coloneqq \{(\bm{x},y)\in \Omega | \ell(f(\bm{x},\bm{\theta}),y) - h > 0\}
\end{equation}
for $(\bm{\theta},h) \in\R^n \times \R$. Also, we can employ stochastic online gradient descent to solve the sequence of optimization problems, where every gradient is well-defined almost surely. Specifically, at each time step $t$, we run one-step gradient descent
\begin{equation}\label{eq:cvar-update}
    (\bm{\theta}_{t+1},h_{t+1}) = (\bm{\theta}_t,h_t) - \gamma \widehat{\nabla} L_\alpha^t (\bm{\theta}_t,h_t;\bm{x}_1^t,\ldots,\bm{x}_m^t,y_1^t,\ldots,y_m^t)
\end{equation}
for $t = 1,\ldots,T-1$, where the gradient approximation is given by
\begin{equation*}
    \widehat{\nabla} L_\alpha^t (\bm{\theta},h;\bm{x}_1^t,\ldots,\bm{x}_m^t,y_1^t,\ldots,y_m^t) = 
    \begin{bmatrix}
    \frac{1}{\alpha} \cdot\frac{1}{m}\sum_{i=1}^m \{\bm{1}_{\ACal(\bm{\theta},h)}(\bm{x}_i^t,y_i^t)\nabla_{\bm \theta} \ell(f(\bm{x}_i^t,\bm{\theta}),y_i^t) \}\\
    -\frac{1}{\alpha}\cdot\frac{1}{m}\sum_{i=1}^m\bm{1}_{\ACal(\bm{\theta},h)}(\bm{x}_i^t,y_i^t) + 1
    \end{bmatrix}.
\end{equation*}
It can be seen that $\E[\widehat{\nabla} L_\alpha^t (\bm{\theta},h;\bm{x}_1^t,\ldots,\bm{x}_m^t,y_1^t,\ldots,y_m^t)] = \nabla L_\alpha^t (\bm{\theta}, h)$.

The recent results in \cite[Lemma 1]{Kalo20} show that if the loss satisfies the set-restricted PL inequality relative to the multifunction $\ACal$ (which will be defined in~\eqref{eq:set-PL}), then the objective function $L_\alpha^t$ satisfies the ordinary PL inequality for $t=1,\ldots,T$. While the PL condition in \cite{Kalo20} was proved over the subset $\Delta' \coloneqq \{(\bm{\theta},h) \colon \PCal^t(\ACal(\bm{\theta},h)) > \alpha + 2\alpha\mu(h_t^* - h)_+\}$, our discussion in Section~\ref{sec:cvar-prelim} shows that $\PCal^t(\ACal(\bm{\theta}_t^*,h_t^*))<\alpha$, implying that an optimum $(\bm{\theta}_t^*,h_t^*)$ does not lie in the subset $\Delta'$. Hence, in the next lemma, we propose a new subset $\Delta$ that $L_\alpha^t$ satisfies the PL condition of $L_\alpha^t$, for $t=1,\ldots,T$, which is much more useful in studying the convergence around an optimum point.

\begin{lemma}[$L_\alpha^t$ is Polyak-Lojasiewicz]\label{lem:CVaR-PL}
    Fix an $\alpha \in (0,1]$. Suppose that for $t=1,\ldots,T$, the following holds:
    \begin{itemize}
        \item[(i)] $\arg\min_{(\bm{\theta},h)\in \R^n \times \R} L_\alpha^t (\bm{\theta},h)\neq \emptyset$ and denote $(\bm{\theta}_t^*,h_t^*)\in\arg\min_{(\bm{\theta},h)}L_\alpha^t(\bm{\theta},h)$;
        \item[(ii)] Let
        \begin{equation}\label{eq:cvar-region}
            \Delta_t \coloneqq \{(\bm{\theta},h)\colon\R^n\times\R \colon \lambda\alpha \le \PCal^t(\ACal(\bm{\theta},h)) \le \alpha + 2\alpha\mu(h_t^* - h)\}.
        \end{equation}
        The loss $\ell(f(\bm{x},\cdot),y)$ satisfies the $\ACal$-restricted PL inequality with parameter $\mu > 0$, relative to $\Omega$ and on $\Delta_t$; i.e.,
        \begin{equation}\label{eq:set-PL}
            \frac{1}{2}\|\E\{\nabla_{\bm \theta}\ell(f(\bm{x},\bm{\theta}),y)|\ACal(\bm{\theta},h)\} \|_2^2 \ge \mu \E \{ \ell(f(\bm{x},\bm{\theta}),y) - \ell^*(\bm{\theta},h)| \ACal(\bm{\theta},h)\}
        \end{equation}
        for all $(\bm{\theta},h)\in\Delta_t$, where $\ell^*(\bullet,\cdot) = \inf_{\tilde{\bm \theta}\in\R^n} \E\{\ell(f(\bm{x},\tilde{\bm \theta}),y)|\ACal(\bullet,\cdot)\}$.
    \end{itemize} 
    Suppose that there exists $0<\lambda<1$ such that for all $t=1,\ldots,T$, it holds that
    \[
    \PCal^t(\ACal(\bm{\theta}_t^*,h_t^*))\ge \lambda\alpha.
    \]
    Then, the CVaR objective $L_\alpha^t$ obeys
        \begin{equation*}            
        \kappa(L_\alpha^t(\bm{\theta},h) - L_\alpha^t(\bm{\theta}_t^*,h_t^*)) \le \frac{1}{2}\| \nabla L_\alpha^t(\bm{\theta},h)\|_2^2
        \end{equation*}
        everywhere on $\Delta_t$, where $\kappa=\lambda\mu$.
\end{lemma}
\begin{proof}
    Recall the definition of  $\ell_\alpha$ in \eqref{eq:ell-alpha}. Adapting the proof in \cite[Lemma 1]{Kalo20}, we have, for every $(\bm{x},y)\in\Omega$,
    \begin{align*}
        &~\ell_\alpha(\bm{\theta},h;\bm{x},y) - \ell_\alpha(\bm{\theta}_t^*,h_t^*;\bm{x},y)\\
        & = h - h_t^* + \frac{1}{\alpha}(\ell(f(\bm{x},{\bm \theta}),y) - h)_+  - \frac{1}{\alpha}(\ell(f(\bm{x},{\bm \theta}_t^*),y) - h_t^*)_+ \\
        &\le h - h_t^* + \frac{1}{\alpha}(\ell(f(\bm{x},{\bm \theta}),y) - h)_+ - \frac{1}{\alpha}\bm{1}_{\ACal(\bm{\theta},h)}(\bm{x},y)(\ell(f(\bm{x},{\bm \theta}_t^*),y) - h_t^*)\\
        &=  h - h_t^* + \frac{1}{\alpha}\bm{1}_{\ACal(\bm{\theta},h)}(\bm{x},y)(\ell(f(\bm{x},{\bm \theta}),y) - \ell(f(\bm{x},{\bm \theta}_t^*),y) + h_t^* - h )\\
        &= (h_t^* - h)\left(\frac{1}{\alpha}\bm{1}_{\ACal(\bm{\theta},h)}(\bm{x},y) - 1\right) + \frac{1}{\alpha}\bm{1}_{\ACal(\bm{\theta},h)}(\bm{x},y)(\ell(f(\bm{x},{\bm \theta}),y) - \ell(f(\bm{x},{\bm \theta}_t^*),y)).
    \end{align*}
    Taking expectation on both sides, it follows that
    \begin{align*}
    &~L_\alpha^t(\bm{\theta},h) - L_\alpha^t(\bm{\theta}_t^*,h_t^*)\\
        &\le (h_t^* - h)\left(\frac{1}{\alpha}\PCal^t(\ACal(\bm{\theta},h)) - 1\right) + \frac{1}{\alpha}\E_{(\bm{x},y)\sim\PP_t}\left\{\bm{1}_{\ACal(\bm{\theta},h)}(\bm{x},y)(\ell(f(\bm{x},{\bm \theta}),y)  - \ell(f(\bm{x},{\bm \theta}_t^*),y))\right\}\\
        &=  (h_t^* - h)\left(\frac{1}{\alpha}\PCal^t(\ACal(\bm{\theta},h)) - 1\right) +\frac{1}{\alpha}\E_{(\bm{x},y)\sim\PP_t}\left\{(\ell(f(\bm{x},{\bm \theta}),y) - \ell(f(\bm{x},{\bm \theta}_t^*),y)) | \ACal(\bm{\theta},h) \right\}\PCal^t(\ACal(\bm{\theta},h)) \\
        &= (h_t^* - h)\left(\frac{1}{\alpha}\PCal^t(\ACal(\bm{\theta},h)) - 1\right)+ \frac{1}{\alpha}\left(\E_{(\bm{x},y)\sim\PP_t}\{(\ell(f(\bm{x},{\bm \theta}),y) | \ACal(\bm{\theta},h)\} - \E\{\ell(f(\bm{x},{\bm \theta}_t^*),y))| \ACal(\bm{\theta},h)\} \right)\PCal^t(\ACal(\bm{\theta},h)) \\
        &\le (h_t^* - h)\left(\frac{1}{\alpha}\PCal^t(\ACal(\bm{\theta},h)) - 1\right) + \frac{1}{\alpha}\left(\E_{(\bm{x},y)\sim\PP_t}\{(\ell(f(\bm{x},{\bm \theta}),y) | \ACal(\bm{\theta},h)\} -\ell_t^*(\bm{\theta},h)\right)\PCal^t(\ACal(\bm{\theta},h))\\
        &= (h_t^* - h)\left(\frac{1}{\alpha}\PCal^t(\ACal(\bm{\theta},h)) - 1\right) + \frac{1}{\alpha}\left(\E_{(\bm{x},y)\sim\PP_t}\{(\ell(f(\bm{x},{\bm \theta}),y) - \ell_t^*(\bm{\theta},h)| \ACal(\bm{\theta},h)\}\right)\PCal^t(\ACal(\bm{\theta},h)).
    \end{align*}
    Therefore, from the set-restricted PL inequality~\eqref{eq:set-PL}, we get
    \begin{align}
        L_\alpha^t(\bm{\theta},h) - L_\alpha^t(\bm{\theta}_t^*,h_t^*) &\le (h_t^* - h)\left(\frac{1}{\alpha}\PCal^t(\ACal(\bm{\theta},h)) - 1\right) + \frac{1}{2\mu\alpha}\|\E\{\nabla_{\bm \theta} \ell(f(\bm{x},\bm{\theta}),y)|\ACal(\bm{\theta},h)\}\|^2\PCal^t(\ACal(\bm{\theta},h)). \nonumber
    \end{align}
    Now, recall the gradient of $L_\alpha^t$ given in \eqref{eq:cvar-grad}. Using the fact that $\PCal^t(\ACal(\bm{\theta}_t^*,h_t^*))<\alpha$ and the definition of $\Delta$, we have
    \[
    \lambda\mu(h_t^* - h)\left(\frac{1}{\alpha}\PCal^t(\ACal(\bm{\theta},h)) - 1\right) \le \left(1 - \frac{1}{\alpha}\PCal^t(\ACal(\bm{\theta},h))\right)^2.
    \]
    The lemma then follows from simple computation.    
\end{proof}

Although set-restricted PL inequality is a new notion in the literature, it is shown that if the loss $\ell(f(\bm{x},\bm{\theta}),y)$ is smooth and strongly convex for $\PCal^t$-almost all $(\bm{x},y)$, then for all events $\BCal$ on the support set, every pair of $(\bm{\theta},\BCal)$ satisfies the set-restricted PL inequality~\cite[Proposition 1]{Kalo20}. Moreover, the next lemma shows some nice properties of $\ell_\alpha$.

\begin{lemma}[Properties of $\ell_\alpha$]\label{lem:ell_alpha}
Fix $\alpha\in(0,1]$. Suppose that
\begin{itemize}
    \item[(i)] Assumption~\ref{assum:exp-diff} holds; and
    \item[(ii)] $\ell(f(\bm{x},\bm{\theta}),y)$ is $K$-Lipschitz continuous wrt $(\bm{x},y)$.
\end{itemize}
Then, the following statements hold:
    \begin{itemize}
        \item[(a)] Given any $(\bm{\theta},h)$, $\ell_\alpha(\bm{\theta}_1,h_1;\bm{x},y)$ is differentiable at $(\bm{\theta},h)$ for almost every $(\bm{x},y)\in \Omega$;
        \item[(b)] $\ell_\alpha(\bm{\theta}_1,h_1;\bm{x},y)$ is locally Lipschitz wrt $(\bm{\theta},h)$;
        \item[(c)] $\ell_\alpha(\bm{\theta}_1,h_1;\bm{x},y)$ is $K$-Lipschitz wrt $(\bm{x},y)$ on $\Omega$.
    \end{itemize}
\end{lemma}

\begin{proof}
Let us prove the statements one by one.
    \begin{itemize}
        \item[(a)] Differentiability of $\ell_\alpha(\bm{\theta}_1,h_1;\bm{x},y)$ at $(\bm{\theta},h)$ for almost every $(\bm{x},y)\in \Omega$ follows directly from Assumption~\ref{assum:exp-diff}. 
    \item[(b)] Suppose that $\ell(f(\bm{x},\bm{\theta}_1),y) - h_1=\epsilon_1$ and $\ell(f(\bm{x},\bm{\theta}_2),y) - h_2=\epsilon_2$. The statement follows directly when $\epsilon_1,\epsilon_2\ge0$ or $\epsilon_1,\epsilon_2<0$. Now, consider $\epsilon_1\ge0$ and $\epsilon_2<0$. Then,
    \begin{align*}
        |\ell_\alpha(\bm{\theta}_1,h_1;\bm{x},y) - \ell_\alpha(\bm{\theta}_2,h_2;\bm{x},y)|  &= \frac{1}{\alpha}|\ell(f(\bm{x},\bm{\theta}_1),y) - h_1| \\
        &= \frac{1}{\alpha} \epsilon_1\\
        &\le \frac{1}{\alpha}(\epsilon_1 - \epsilon_2)\\
        &= \frac{1}{\alpha}(\ell(f(\bm{x},\bm{\theta_1}),y) - \ell(f(\bm{x},\bm{\theta_2}),y)) + \frac{1}{\alpha}(h_2 - h_1).
    \end{align*}
    The local Lipschitzness of $\ell_\alpha$ wrt $(\bm{\theta},h)$ then follows from Assumption~\ref{assum:exp-diff}.
    \item[(c)] Following the trick in the above argument, suppose that $\ell(f(\bm{x}_1,\bm{\theta}),y_1) - h=\epsilon_1$ and $\ell(f(\bm{x}_2,\bm{\theta}),y_2) - h =\epsilon_2$. It remains to consider the case that $\epsilon_1>0$ and $\epsilon_2<0$. Then,
    \begin{align*}
        |\ell_\alpha(\bm{\theta},h;\bm{x}_1,y_1) - \ell_\alpha(\bm{\theta},h;\bm{x}_2,y_2)|  &= \frac{1}{\alpha}|\ell(f(\bm{x}_1,\bm{\theta}),y_1) - h| \\
        &= \frac{1}{\alpha} \epsilon_1\\
        &\le \frac{1}{\alpha}(\epsilon_1 - \epsilon_2)\\
        &= \frac{1}{\alpha}(\ell(f(\bm{x}_1,\bm{\theta}),y_1) - \ell(f(\bm{x}_2,\bm{\theta}),y_2)),
    \end{align*}
    which leads to the Lipschitzness result given assumption (ii).
    \end{itemize}
\end{proof}

Having the above lemmas, we are ready to apply our framework to the CVaR problem.
\begin{corollary}\label{cor:cvar}
Fix $\alpha\in(0,1]$. Under the setting of Lemma~\ref{lem:CVaR-PL}, suppose that
\begin{itemize}
    \item[(i)] assumptions (i) and (ii) in Lemma~\ref{lem:ell_alpha} hold;
    \item[(ii)] every underlying distribution has a bounded support set;
    \item[(iii)] the probability density function of every distribution is differentiable;
    \item[(iv)] the Wasserstein distance of any two successive distributions is bounded; i.e.,
    \[
    \mathfrak{M}(\PP_{t+1}, \PP_t) \le \eta_t,\quad{\rm for}~t=1,\ldots,T-1;
    \]
    \item[(v)] the variance of the gradient approximation is upper bounded by 
    \[
    \E[\|\widehat{\nabla} L_\alpha^t(\bm{\theta},h;\bm{x}_1^t,\ldots,\bm{x}_m^t,y_1^t,\ldots,y_m^t) - \nabla L_\alpha^t(\bm{\theta},h)\|^2] \le \sigma_t^2
    \]
    for some $\sigma_t>0$ and for $t=1,\ldots,T$; and
    \item[(vi)] $L_\alpha^t$ is $\beta$-smooth on $\Delta_t$ for $t=1,\ldots,T$.    
\end{itemize}
Suppose that the step size $\gamma_t\equiv \gamma\in(0,1/(2\kappa))$ for $t=1,\ldots,T$. If the iterates $(\bm{\theta}_t,h_t)\in \Delta_t$ over all $t=1,\ldots,T$, writing $\zeta = -\frac{\gamma^2\beta}{2} +\gamma$, a regret bound for stochastic online gradient descent satisfies
\begin{equation*}
    {\rm Regret}(T)\le \frac{1}{2\kappa\zeta}(L_\alpha^1(\bm{\theta}_1,h_1) - (L_\alpha^1)^*) + \frac{1}{\kappa\zeta} \left(K+\frac{C}{4\kappa}\right) \sum_{t=1}^T \eta_t + \frac{\gamma\beta}{2\kappa} \sum_{t=1}^{T-1} \sigma_t^2,
\end{equation*}
where $(L_\alpha^1)^* = \min_{(\bm{\theta},h)}L_\alpha^1(\bm{\theta},h)$ and $C>0$ is some constant that depends on the CVaR parameter $\alpha$, the loss function $\ell(f(\cdot,\cdot),\cdot)$, and the probability density functions of the underlying distributions $\{\PP_t\}_{t=1}^T$.
\end{corollary}
\begin{proof}
Let us verify that problem \eqref{eq:cvar-prob}
for $t=1,\ldots,T$ satisfies the assumptions in Theorem~\ref{thm:PL}. Using the results in Lemmas~\ref{lem:CVaR-PL} and~\ref{lem:ell_alpha}, it remains to show that
\[
\|  \E_{\bm{(\bm{x},\bm{y})}\sim\PP_{t+1}}[\nabla \ell_\alpha(\bm{\theta},h;\bm{x},y)] - \E_{\bm{(\bm{x},\bm{y})}\sim\PP_t}[\nabla \ell_\alpha(\bm{\theta},h;\bm{x},y)] \| \le C\sqrt{\eta_t}
\]
for some $C>0$. Recall that 
    \begin{equation*}
        \nabla_{(\bm{\theta},h)} \ell_\alpha(\bm{\theta},h;\bm{x},y)  = 
        \begin{bmatrix}
            \frac{1}{\alpha}\bm{1}_{\ACal(\bm{\theta},h)}(\bm{x},y)\nabla_{\bm \theta} \ell(f(\bm{x},\bm{\theta}),y) \\
    -\frac{1}{\alpha}\bm{1}_{\ACal(\bm{\theta},h)}(\bm{x},y) + 1
        \end{bmatrix}.
    \end{equation*}
    Given $(\bm{\theta},h) \in \R^n \times \R$, we see that $\nabla_{\bm \theta} \ell (f(\bm{x},\bm{\theta}),y)$ is bounded on the support set $\Omega$, due to the assumptions (v) and (vi). Assume that  $\|\nabla_{\bm \theta} \ell (f(\bm{x},\bm{\theta}),y)\| \le M$ for some $M>0$. Then,
    \begin{align*}
        &\left\| \E_{(\bm{x},y)\sim\PP_{t+1}} \left[\frac{1}{\alpha} \bm{1}_{\ACal(\bm{\theta},h)}(\bm{x},y) \nabla_{\bm \theta} \ell (f(\bm{x},\bm{\theta}),y)\right] - \E_{(\bm{x},y)\sim\PP_t}\left[\frac{1}{\alpha} \bm{1}_{\ACal(\bm{\theta},h)}(\bm{x},y) \nabla_{\bm \theta} \ell (f(\bm{x},\bm{\theta}),y)\right]\right\| \\
        \le& \frac{M}{\alpha}|\PCal^{t+1}(\ACal(\bm{\theta},h)) - \PCal^t(\ACal(\bm{\theta},h))|.
    \end{align*}
    Also,
    \begin{align*}
        \left| \E_{(\bm{x},y)\sim\PP_{t+1}} \left[ -\frac{1}{\alpha}\bm{1}_{\ACal(\bm{\theta},h)}(\bm{x},y) + 1 \right] - \E_{(\bm{x},y)\sim\PP_t} \left[ -\frac{1}{\alpha}\bm{1}_{\ACal(\bm{\theta},h)}(\bm{x},y) + 1 \right]\right| \le \frac{1}{\alpha} |\PCal^{t+1}(\ACal(\bm{\theta},h)) - \PCal^t(\ACal(\bm{\theta},h))|.
    \end{align*}
    It remains to bound $|\PCal^{t+1}(\ACal(\bm{\theta},h)) - \PCal^t(\ACal(\bm{\theta},h))|$. Now, let us invoke a theorem from \cite{CW20}.
    \begin{lemma}[c.f. {\cite[Theorem 2.1]{CW20}}]
        Let $p_t$ and $p_{t+1}$ be the probability density function of the distributions $\PP_t$ and $\PP_{t+1}$. Then,
        \begin{equation*}
            \|p_t - p_{t+1}\|_1^2 \le c (\|p_t\|_1 + \|D p_t\|_1 +\|p_{t+1}\|_1 + \|D p_{t+1}\|_1) \cdot \mathfrak{M}(\PP_t,\PP_{t+1})
        \end{equation*}
        for some constant $c>0$, where $D$ is the differential operator and $\|\cdot\|_1$ is the $\ell_1$-norm wrt the Lebesgue measure.
    \end{lemma}
    Let $\mathcal{E}$ be the event space. The theorem implies that the total variation distance $\sup_{\bm{A}\in \mathcal{E}} |\PCal^{t+1}(\bm{A}) - \PCal^t(\bm{A})|$ is upper bounded in terms of $\mathfrak{M}(\PP_t,\PP_{t+1})$, since
    \begin{align*}
         \sup_{\bm{A}\in \mathcal{E}} |\PCal^{t+1}(\bm{A}) - \PCal^t(\bm{A})|  &= \left| \int_{\bm A} p_t(\bm{x},y) d(\bm{x},y) - \int_{\bm A} p_{t+1}(\bm{x},y) d(\bm{x},y)\right| \\
         &\le \int_{\bm A} |p_t(\bm{x},y) - p_{t+1}(\bm{x},y)| d(\bm{x},y) \\
         &\le \int_{\Omega} |p_t(\bm{x},y) - p_{t+1}(\bm{x},y)| d(\bm{x},y)
          = \|p_t - p_{t+1}\|_1.
    \end{align*}
Consequently, applying Theorem~\ref{thm:PL} yields the desired result.
\end{proof}
Corollary~\ref{cor:cvar} shows that, under assumptions (i)--(vi) in Corollary~\ref{cor:cvar}, the regret of online stochastic gradient descent grows sublinearly when both the cumulative distribution drifts and the cumulative gradient noise variances grow sublinearly. In particular, the assumption on the smoothness of $L_\alpha$ is shown to be satisfied if the gradient on $(\bm{x},y)$ is not zero on the boundary of the event set~\cite[Section 2]{AMRU01}; for details on the assumption see~\cite[Theorem 2.1]{Ury95}. Although the conditions are described as general in~\cite{AMRU01,Ury00,YKRW11}, the conditions on the smoothness could be hard to verify. A number of works suggest smooth approximation of the CVaR problem; see, e.g.,~\cite{Kalo20,SY20}. Similar analysis could be applied but a cumulative approximation error term would be involved in the regret bound. 

\begin{remark}
    When a regularizer is added to the CVaR formulation, it is not clear whether \emph{set-restricted proximal PL inequality} (an analogy to proximal PL inequality) of $\ell + R$ would lead to the proximal PL condition of the regularized CVaR objective $L_\alpha^t + R$, for some regularizer $R$. The main technical difficulty lies in comparing the minimum values involved in the proximal PL inequality and the set-restricted proximal inequality when a regularizer exists. One may need to explore whether set-restricted proximal PL inequality is still a suitable tool to understand the proximal PL condition of the regularized CVaR learning problem. We will leave this as a future work.
\end{remark}

\section{Numerical Simulations}\label{sec:sim}
In this section, we present some numerical results to illustrate the theoretical findings of our proposed framework. Specifically, in the following, at every time step $t$ (for $t=1,\ldots,T$), we generate the set of data $\{(\bm{u}^{i,t},d_i^t)\}_{i=1}^m$, where
\[
d_i^t = \tilde{\bm{\theta}}_t^T \bm{u}^{i,t} + \nu_i^t.
\]
Here, $\bm{u}^{i,t}\sim\NCal(\bm{0},\bm{I})$ is a random vector with dimension $n=5$, where every entry follows an independent and identically distributed (iid) Gaussian distribution with zero mean; $\nu_i^t\sim\NCal(0,0.5)$ is some mean-zero measurement noise with variance 0.5; and $T=500$ is the horizon length of interest. For $t=1,\ldots,T-1$, $\tilde{\bm{\theta}}_t\in\R^n$ is deterministic, unknown and time-varying, which we initialize at $\tilde{\bm \theta}_1 = \bm{e}$ and update by
\begin{equation}\label{eq:sim-update}
    \tilde{\bm{\theta}}_{t+1} = {\rm proj}_C (\tilde{\bm{\theta}}_t + \bm{z}^t)
\end{equation}
with $\bm{e}\in\R^n$ being the all-one vector, $\bm{z}^t\sim\NCal(\bm{0},10^{-4}\cdot t^{-1}\bm{I})$ and some convex set $C\subseteq\R^n$ in the numerical simulations. We assess the performance of online stochastic gradient descent (resp. online stochastic proximal gradient descent) when the objective function is unconstrained (resp. constrained or regularized) via \emph{relative regret}, which is given by~\cite[Section IV]{CZP21}
\[
{\rm Relative~regret}(t) = \frac{1}{t}\cdot \frac{{\rm Regret}(t)}{{\rm Regret}(1)}.
\]
The relative regret shown in the figures are averaged over 100 Monte Carlo runs.

\subsection{Adaptive Filtering}
In this example, we are interested in solving the adaptive filtering problem, which can be posed as an online stochastic optimization problem with time-varying distributions~\cite{CZP21}:
\begin{equation}\label{eq:sim-lr}
    \inf_{\bm{\theta}\in\R^n} \E_{(\bm{u},d)\sim\PP_t}[(d - \bm{\theta}^T\bm{u})^2]+R(\bm{\theta})
\end{equation}
for $t=1,\ldots,T$. We consider three optimization problems corresponding to different regularizers and different feasible set $C$ defined in \eqref{eq:sim-update}: (i) an unconstrained optimization problem, where $C=\R^n$ and $R=0$; (ii) a constrained optimization problem, where $C=[-5,5]^n$ and $R(\cdot)=\bm{1}_C(\cdot)$ with $\bm{1}_C$ as the indicator function wrt $C$; and (iii) a regularized optimization problem, where $C=\R^n$ and $R(\cdot)=\|\cdot\|_1$. We apply online stochastic gradient descent for problem (i) and apply online stochastic proximal gradient descent for problems (ii) and (iii), all with initialization $\bm{\theta}_1 = \bm{0}$. We test the performance of both methods using two different step sizes: (a) a constant step size $\gamma_t = 0.01/\sqrt{T}$, and (b) a decaying step size $\gamma_t = 0.01/\sqrt{t}$ for $t=1,\ldots,T-1$. The number of samples drawn at each time step is $m=5$. When applying online stochastic proximal gradient descent, we use the fact that, for any $\bm{x}\in\R^n$, the proximal step for $R(\cdot) = \bm{1}_C(\cdot)$ is given by
\begin{equation*}
    ({\rm prox}_{\gamma_t R} (\bm{x}))_i =
    \begin{cases}
        -5,\quad& x_i>5\\
        x_i,\quad &x_i\in[-5,5]\\
        5,\quad &x_i<-5
    \end{cases}
\end{equation*}
while that for $R(\cdot) = \|\cdot\|_1$ is given by
\[
({\rm prox}_{\gamma_t R} (\bm{x}))_i = {\rm sgn}(x_i) \max\{|x_i| - \gamma_t,0\}
\]
for $i=1,\ldots,m$. We need to find an optimal point of each problem to compute a relative regret at each time step. For problems (i) and (ii), it is known that an optimal point at time $t$ is given by $\bm{\theta}_t^* = \tilde{\bm \theta}_t$~\cite{CZP21}. For problem (iii), we use the true vector $\tilde{\bm \theta}_t$ as the initial point and perform the proximal gradient descent updates using the constant step size 0.01 until either the difference of the objective values of successive iterates is less than $10^{-6}$ or the number of iterations reaches 1000. We then declare it as an optimal point $\bm{\theta}_t^*$.

\begin{figure}
\begin{subfigure}[b]{0.5\textwidth}
    \centering
    \includegraphics[width=\textwidth]{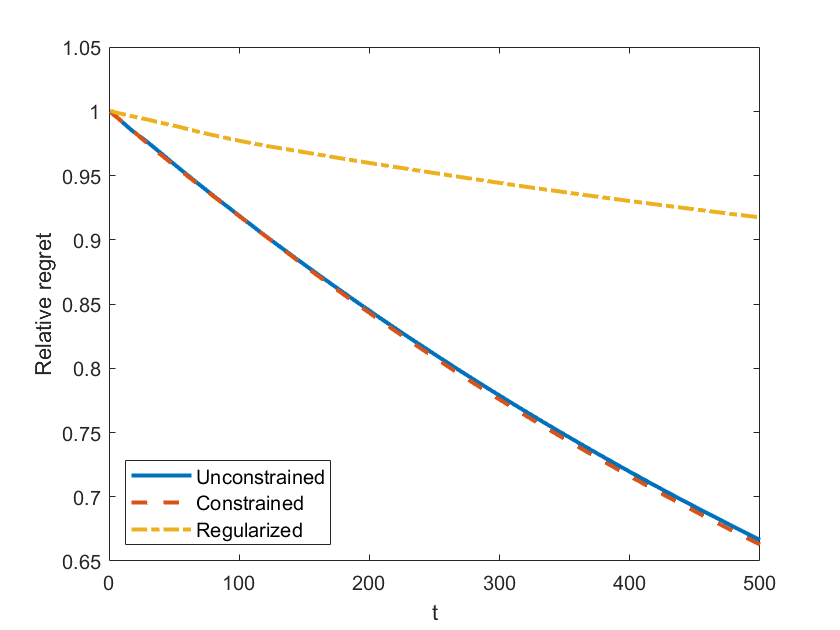}
    \caption{$\gamma_t = 0.01/\sqrt{T}$.}
    \label{fig:LR_sqrtT}
\end{subfigure}
\begin{subfigure}[b]{0.5\textwidth}
    \centering
    \includegraphics[width=\textwidth]{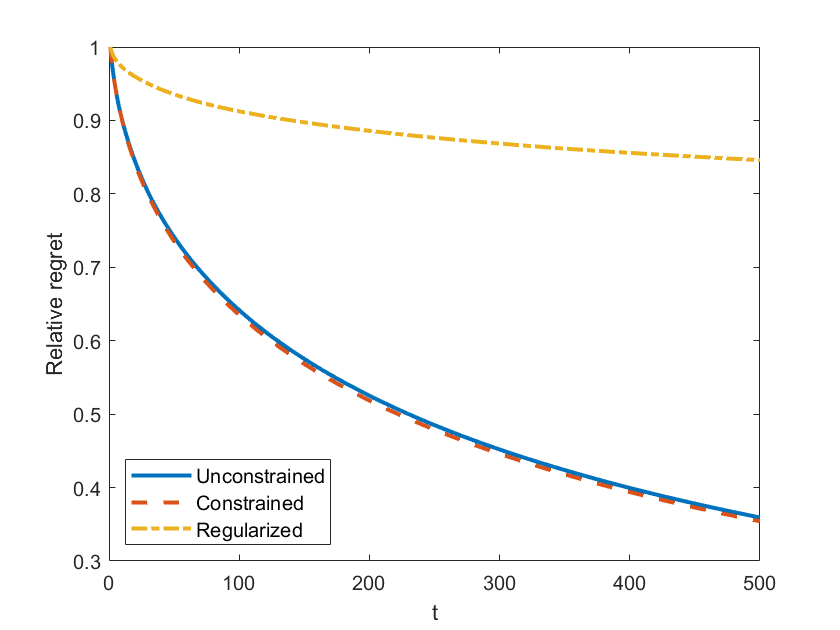}
    \caption{$\gamma_t = 0.01/\sqrt{t}$.}
    \label{fig:LR_sqrtt}
\end{subfigure}    
    \caption{Relative regret of online stochastic gradient descent and online stochastic proximal gradient descent when the adaptive filtering problem is unconstrained, constrained or regularized.}
    \label{fig:LR}
\end{figure}

Figure~\ref{fig:LR_sqrtT} shows the relative regret of online stochastic gradient descent and online stochastic proximal gradient descent with a constant step size $\gamma_t = 0.01/\sqrt{T}$ for all $t$ when the adaptive filtering problem is unconstrained and constrained/regularized, respectively. As can be seen, the relative regret of online stochastic gradient descent applying to the unconstrained problem decreases when $t$ increases, implying a sublinear regret of online stochastic gradient descent. This verifies our findings in Theorem~\ref{thm:PL}. Despite the fact that the cumulative variance of the measurement noise grows linearly, a sublinear regret of online stochastic gradient descent can be achieved given a suitable step size rule. Similar results can be observed for online stochastic gradient descent. Specifically, although Theorem~\ref{thm:prox-PL} cannot guarantee a sublinear regret bound of online stochastic proximal gradient descent as discussed in Remark~\ref{rmk:prox-ss}, we see that a sublinear regret bound can be achieved in numerical simulations when the adaptive filtering problem is either constrained or regularized.

Figure~\ref{fig:LR_sqrtt} shows the relative regret of online stochastic gradient descent and online stochastic proximal gradient descent with a decaying step size $\gamma_t = 0.01/\sqrt{t}$ for $t=1,\ldots,T$ when applied to the adaptive filtering problem with different regularizers. As can be seen, the online stochastic gradient descent (resp. online stochastic proximal gradient descent) achieves sublinear regret when the problem is unconstrained (resp. constrained or regularized). This verifies our discussion in Remark~\ref{rmk:stepsize} that the step size can be set to be decreasing instead of constant. Moreover, as the step size is larger at the beginning, the learning rate is faster than that using constant step size, resulting in a lower relative regret of both online stochastic gradient descent and online stochastic proximal gradient descent given different regularizers. Besides, using either step size, we see that the relative regret of online stochastic proximal gradient when applied to the regularized problem decreases at the slowest speed. This partly explains the technical difficulty in improving the regret bound proved in Theorem~\ref{thm:prox-PL} that the structure of the regularizer could seriously affect the performance of the online algorithms.

\subsection{CVaR Learning}
In this example, we consider the online CVaR learning problem with time-varying distribution:
\begin{equation}\label{eq:sim-cvar}
    \inf_{\bm{\theta},h} \E_{(\bm{u},d)\sim\PP_t} \left[h + \frac{1}{\alpha}((d - \bm{\theta}^T\bm{u})^2-h)_+\right]+R(\bm{\theta})
\end{equation}
with $\alpha = 0.95$. Using the same setting as in the previous example, we consider all unconstrained, constrained and regularized optimization problems of \eqref{eq:sim-cvar}. To better estimate the underlying probability distribution, we draw $m=20$ samples drawn at each time step. We apply online stochastic gradient descent for problem (i) and apply online stochastic proximal gradient descent for problems (ii) and (iii), all with initialization $(\bm{\theta}_1,h_1) = \bm{0}$. We test the performance of both methods with the following two step sizes: (a) a constant step size $\gamma_t = 0.01/\sqrt{T}$, and (b) a decaying step size $\gamma_t = 0.01/\sqrt{t}$ for $t=1,\ldots,T-1$. An optimal point for computing a relative regret is found as follows: At each time step, we approximate the distribution using a new sample set with 100 samples. Then, for all unconstrained, constrained and regularized versions of problem \eqref{eq:sim-cvar}, we initialize the iterate at the origin and perform the gradient descent (or proximal gradient descent) updates using the constant step size 0.01 until either the difference of the objective values of successive iterates is less than $0.01$ or the number of iterations reaches 1000. We then declare it as an optimal point $(\bm{\theta}_t^*,h_t^*)$.

\begin{figure}
\begin{subfigure}[b]{0.5\textwidth}
    \centering
    \includegraphics[width=\textwidth]{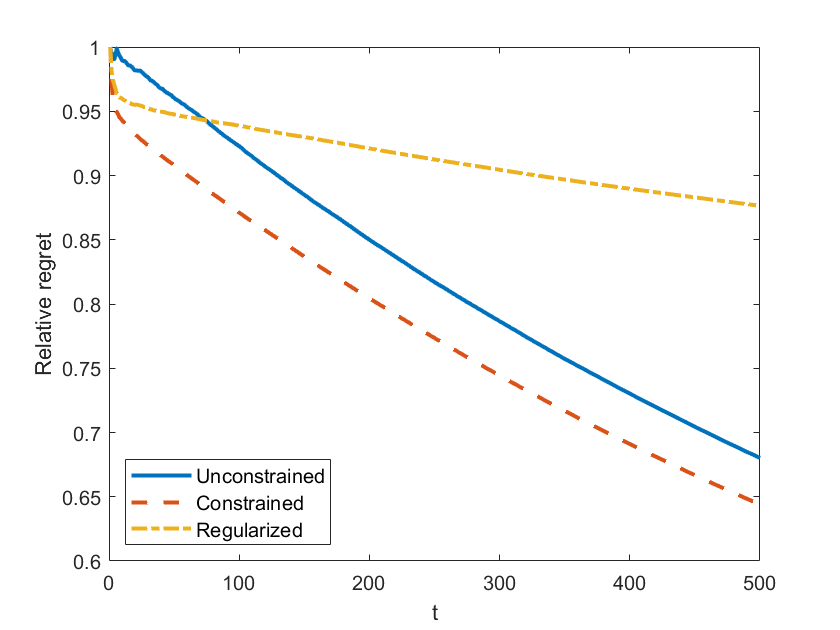}
    \caption{$\gamma_t=0.01\sqrt{T}$.}
    \label{fig:cvar_sqrtT}
\end{subfigure}
\begin{subfigure}[b]{0.5\textwidth}
    \centering
    \includegraphics[width=\textwidth]{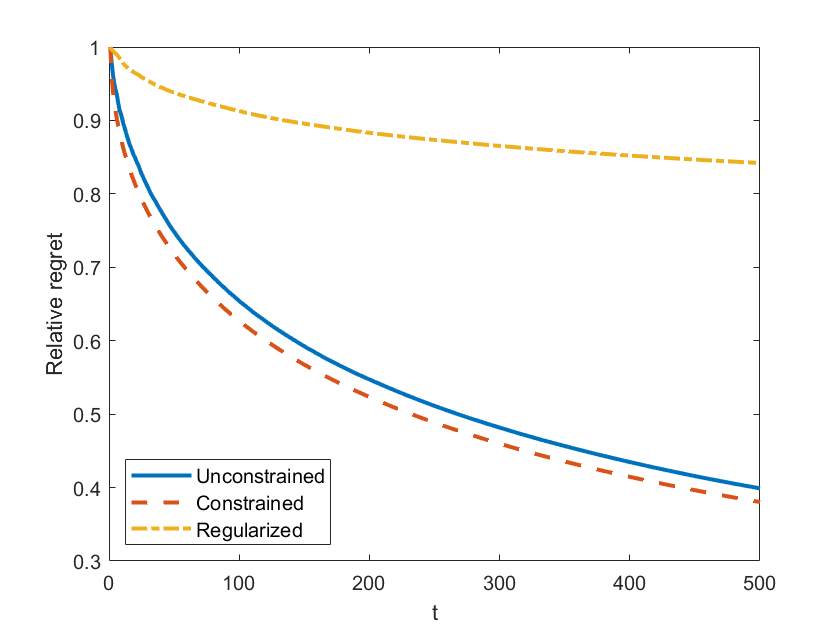}
    \caption{$\gamma_t=0.01\sqrt{t}$.}
    \label{fig:cvar_sqrtt}
\end{subfigure}    
    \caption{Relative regret of online stochastic gradient descent and online stochastic proximal gradient descent when the online CVaR learning problem is unconstrained, constrained or regularized.}
    \label{fig:cvar}
\end{figure}
Figure~\ref{fig:cvar_sqrtT} shows the relative regret of online stochastic gradient descent and online stochastic proximal gradient descent with a constant step size $\gamma_t = 0.01/\sqrt{T}$ for all $t$. It can be seen that both online stochastic gradient descent and online stochastic proximal gradient descent enjoy sublinear regret regardless of the regularizers. This matches our result in Corollary~\ref{cor:cvar} that online stochastic gradient descent achieves sublinear regret when applied to unconstrained online CVaR problem. Although it is not known whether a regularized CVaR learning problem possesses proximal PL condition, we see that online stochastic proximal gradient descent achieves sublinear regret when applied to constrained or regularized version of \eqref{eq:sim-cvar}. In particular, we see that the relative regret of online stochastic gradient descent when applied to the unconstrained problem and that of online stochastic proximal gradient descent when applied to the constrained problem decrease almost at the same rate, while the relative regret of online stochastic proximal gradient descent when applied to the regularized problem decreases at the slowest speed. This is because the $\ell_1$ regularizer destroys the smoothness of the problem, resulting in a slower convergence of the algorithm. On the other hand, the online stochastic proximal gradient descent performs better than the online stochastic gradient descent when the problem is constrained, because more knowledge on the underlying distribution is available compared with the unconstrained problem.

Figure~\ref{fig:cvar_sqrtt} shows the relative regrets of online stochastic gradient descent and online stochastic proximal gradient descent when applied to the unconstrained problem and constrained/regularized problem, respectively. Similar to Figure~\ref{fig:cvar_sqrtT}, all the curves are decreasing, implying sublinear regrets of both methods when applied to the corresponding problems. Also, we see that the relative regret of the online stochastic proximal gradient descent is the lowest, whereas that when applied to the regularized problem is the highest. Comparing to Figure~\ref{fig:cvar_sqrtT}, we see that both methods perform better using a decaying step size instead of a constant step size, because of the faster learning rate at the beginning.

\section{Conclusion}
In this paper, we considered an online stochastic optimization problem with a time-varying distribution, when the loss function satisfies the PL condition. We established a regret bound of online stochastic gradient descent, which is composed of the cumulative gradient biases caused by stochasticity and the cumulative Wasserstein distances between distribution at consecutive time steps. A similar regret bound of online stochastic proximal gradient descent was also shown when the objective function is regularized. We applied this framework to the CVaR learning problem by improving an existing proof of its PL condition and established its regret bound. Our numerical results support our theoretical findings and demonstrate the power of the framework. An interesting future direction is to apply the said framework to other data-driven modeling optimization problems with time-varying distribution. Particularly, it is intriguing to see under what condition the CVaR problem possesses proximal PL condition when it is regularized.

\bibliographystyle{plain}
\bibliography{reference}

\begin{thebibliography}{10}

\bibitem{AEK15}
Soroosh~Shafieezadeh Abadeh, Peyman~Mohajerin Esfahani, and Daniel Kuhn.
\newblock Distributionally robust logistic regression.
\newblock In {\em Advances in Neural Information Processing Systems}, pages
  1576--1584, 2015.

\bibitem{ASD19}
Amirhossein Ajalloeian, Andrea Simonetto, and Emiliano Dall'Anese.
\newblock Inexact online proximal-gradient method for time-varying convex
  optimization.
\newblock {\em arXiv preprint arXiv:1910.02018}, 2019.

\bibitem{AMRU01}
Fredrik Andersson, Helmut Mausser, Dan Rosen, and Stanislav Uryasev.
\newblock Credit risk optimization with conditional value-at-risk criterion.
\newblock {\em Mathematical programming}, 89:273--291, 2001.

\bibitem{ABMY14}
Francisco~J Arag{\'o}n~Artacho, Jonathan~M Borwein, Victoria
  Mart{\'\i}n-M{\'a}rquez, and Liangjin Yao.
\newblock Applications of convex analysis within mathematics.
\newblock {\em Mathematical programming}, 148:49--88, 2014.

\bibitem{BBB17}
Alejandro Balb{\'a}s, Beatriz Balb{\'a}s, and Raquel Balb{\'a}s.
\newblock Differential equations connecting var and cvar.
\newblock {\em Journal of Computational and Applied Mathematics}, 326:247--267,
  2017.

\bibitem{Beck17}
Amir Beck.
\newblock {\em First-order methods in optimization}.
\newblock SIAM, 2017.

\bibitem{BSR18}
Amrit~Singh Bedi, Paban Sarma, and Ketan Rajawat.
\newblock Tracking moving agents via inexact online gradient descent algorithm.
\newblock {\em IEEE Journal of Selected Topics in Signal Processing},
  12(1):202--217, 2018.

\bibitem{BDD+13}
Aharon Ben-Tal, Dick Den~Hertog, Anja De~Waegenaere, Bertrand Melenberg, and
  Gijs Rennen.
\newblock Robust solutions of optimization problems affected by uncertain
  probabilities.
\newblock {\em Management Science}, 59(2):341--357, 2013.

\bibitem{BS04}
Dimitris Bertsimas and Melvyn Sim.
\newblock The price of robustness.
\newblock {\em Operations research}, 52(1):35--53, 2004.

\bibitem{BGZ15}
Omar Besbes, Yonatan Gur, and Assaf Zeevi.
\newblock Non-stationary stochastic optimization.
\newblock {\em Operations research}, 63(5):1227--1244, 2015.

\bibitem{CZP21}
Xuanyu Cao, Junshan Zhang, and H.~Vincent Poor.
\newblock Online stochastic optimization with time-varying distributions.
\newblock {\em IEEE Transactions on Automatic Control}, 66(4):1840--1847, 2021.

\bibitem{CW20}
Minwoo Chae and Stephen~G Walker.
\newblock Wasserstein upper bounds of the total variation for smooth densities.
\newblock {\em Statistics \& Probability Letters}, 163:108771, 2020.

\bibitem{CBS+21}
Margaret~P Chapman, Riccardo Bonalli, Kevin~M Smith, Insoon Yang, Marco Pavone,
  and Claire~J Tomlin.
\newblock Risk-sensitive safety analysis using conditional value-at-risk.
\newblock {\em IEEE Transactions on Automatic Control}, 67(12):6521--6536,
  2021.

\bibitem{CW10}
Naveed Chehrazi and Thomas~A Weber.
\newblock Monotone approximation of decision problems.
\newblock {\em Operations Research}, 58(4-part-2):1158--1177, 2010.

\bibitem{Clarke75}
Frank~H Clarke.
\newblock Generalized gradients and applications.
\newblock {\em Transactions of the American Mathematical Society},
  205:247--262, 1975.

\bibitem{CT06}
T.~M. Cover and J.~A. Thomas.
\newblock {\em Elements of Information Theory}.
\newblock John Wiley \& Sons, Hoboken, NJ, 2 edition, 2006.

\bibitem{CDH23}
Joshua Cutler, Dmitriy Drusvyatskiy, and Zaid Harchaoui.
\newblock Stochastic optimization under distributional drift.
\newblock {\em Journal of Machine Learning Research}, 24(147):1--56, 2023.

\bibitem{DY10}
Erick Delage and Yinyu Ye.
\newblock Distributionally robust optimization under moment uncertainty with
  application to data-driven problems.
\newblock {\em Operations research}, 58(3):595--612, 2010.

\bibitem{DBT+19}
Rishabh Dixit, Amrit~Singh Bedi, Ruchi Tripathi, and Ketan Rajawat.
\newblock Online learning with inexact proximal online gradient descent
  algorithms.
\newblock {\em IEEE Transactions on Signal Processing}, 67(5):1338--1352, 2019.

\bibitem{DMP+21}
Omar~Darwiche Domingues, Pierre M{\'e}nard, Matteo Pirotta, Emilie Kaufmann,
  and Michal Valko.
\newblock A kernel-based approach to non-stationary reinforcement learning in
  metric spaces.
\newblock In {\em International Conference on Artificial Intelligence and
  Statistics}, pages 3538--3546. PMLR, 2021.

\bibitem{DL18}
Dmitriy Drusvyatskiy and Adrian~S Lewis.
\newblock Error bounds, quadratic growth, and linear convergence of proximal
  methods.
\newblock {\em Mathematics of Operations Research}, 43(3):919--948, 2018.

\bibitem{EI06}
Emre Erdo{\u{g}}an and Garud Iyengar.
\newblock Ambiguous chance constrained problems and robust optimization.
\newblock {\em Mathematical Programming}, 107(1-2):37--61, 2006.

\bibitem{EK18}
Peyman~Mohajerin Esfahani and Daniel Kuhn.
\newblock Data-driven distributionally robust optimization using the
  wasserstein metric: Performance guarantees and tractable reformulations.
\newblock {\em Mathematical Programming}, 171(1-2):115--166, 2018.

\bibitem{faro20}
Farhad Farokhi.
\newblock Regularization helps with mitigating poisoning attacks:
  Distributionally-robust machine learning using the wasserstein distance.
\newblock {\em arXiv preprint arXiv:2001.10655}, 2020.

\bibitem{Gaivor05}
Alexei~A Gaivoronski.
\newblock Stochastic optimization problems in telecommunications.
\newblock In {\em Applications of stochastic programming}, pages 669--704.
  SIAM, 2005.

\bibitem{Garr23}
Guillaume Garrigos.
\newblock Square distance functions are polyak-$\{$$\backslash$L$\}$ ojasiewicz
  and vice-versa.
\newblock {\em arXiv preprint arXiv:2301.10332}, 2023.

\bibitem{GS10}
Joel Goh and Melvyn Sim.
\newblock Distributionally robust optimization and its tractable
  approximations.
\newblock {\em Operations research}, 58(4-part-1):902--917, 2010.

\bibitem{HK14}
Elad Hazan and Satyen Kale.
\newblock Beyond the regret minimization barrier: optimal algorithms for
  stochastic strongly-convex optimization.
\newblock {\em The Journal of Machine Learning Research}, 15(1):2489--2512,
  2014.

\bibitem{HH13}
Zhaolin Hu and L~Jeff Hong.
\newblock Kullback-leibler divergence constrained distributionally robust
  optimization.
\newblock {\em Available at Optimization Online}, 2013.

\bibitem{JLZ20}
Jiashuo Jiang, Xiaocheng Li, and Jiawei Zhang.
\newblock Online stochastic optimization with wasserstein based
  non-stationarity.
\newblock {\em arXiv preprint arXiv:2012.06961}, 2020.

\bibitem{JG16}
Ruiwei Jiang and Yongpei Guan.
\newblock Data-driven chance constrained stochastic program.
\newblock {\em Mathematical Programming}, 158(1-2):291--327, 2016.

\bibitem{Kalo20}
Dionysios~S Kalogerias.
\newblock Noisy linear convergence of stochastic gradient descent for cv@ r
  statistical learning under polyak-$\{$$\backslash$L$\}$ ojasiewicz
  conditions.
\newblock {\em arXiv preprint arXiv:2012.07785}, 2020.

\bibitem{Kalo22}
Dionysios~S Kalogerias.
\newblock Fast and stable convergence of online sgd for cv@ r-based risk-aware
  learning.
\newblock In {\em ICASSP 2022-2022 IEEE International Conference on Acoustics,
  Speech and Signal Processing (ICASSP)}, pages 6007--6011. IEEE, 2022.

\bibitem{KNS16}
Hamed Karimi, Julie Nutini, and Mark Schmidt.
\newblock Linear convergence of gradient and proximal-gradient methods under
  the polyak-{\l}ojasiewicz condition.
\newblock In {\em Joint European Conference on Machine Learning and Knowledge
  Discovery in Databases}, pages 795--811. Springer, 2016.

\bibitem{KMD22}
Seunghyun Kim, Liam Madden, and Emiliano Dall’Anese.
\newblock Online stochastic gradient methods under sub-weibull noise and the
  polyak-lojasiewicz condition.
\newblock In {\em 2022 IEEE 61st Conference on Decision and Control (CDC)},
  pages 3499--3506. IEEE, 2022.

\bibitem{KPT+17}
Soheil Kolouri, Se~Rim Park, Matthew Thorpe, Dejan Slepcev, and Gustavo~K
  Rohde.
\newblock Optimal mass transport: Signal processing and machine-learning
  applications.
\newblock {\em IEEE signal processing magazine}, 34(4):43--59, 2017.

\bibitem{KESN19}
Daniel Kuhn, Peyman~Mohajerin Esfahani, Viet~Anh Nguyen, and Soroosh
  Shafieezadeh-Abadeh.
\newblock Wasserstein distributionally robust optimization: Theory and
  applications in machine learning.
\newblock In {\em Operations Research \& Management Science in the Age of
  Analytics}, pages 130--166. 2019.

\bibitem{LTS20}
Antoine Lesage-Landry, Joshua~A Taylor, and Iman Shames.
\newblock Second-order online nonconvex optimization.
\newblock {\em IEEE Transactions on Automatic Control}, 66(10):4866--4872,
  2020.

\bibitem{LCS20}
Jiajin Li, Caihua Chen, and Anthony Man-Cho So.
\newblock Fast epigraphical projection-based incremental algorithms for
  wasserstein distributionally robust support vector machine.
\newblock {\em Advances in Neural Information Processing Systems},
  33:4029--4039, 2020.

\bibitem{LSM20}
Jiajin Li, Anthony Man-Cho So, and Wing-Kin Ma.
\newblock Understanding notions of stationarity in nonsmooth optimization: A
  guided tour of various constructions of subdifferential for nonsmooth
  functions.
\newblock {\em IEEE Signal Processing Magazine}, 37(5):18--31, 2020.

\bibitem{LL19}
Yingying Li and Na~Li.
\newblock Online learning for markov decision processes in nonstationary
  environments: A dynamic regret analysis.
\newblock In {\em 2019 American Control Conference (ACC)}, pages 1232--1237.
  IEEE, 2019.

\bibitem{LZB22}
Chaoyue Liu, Libin Zhu, and Mikhail Belkin.
\newblock Loss landscapes and optimization in over-parameterized non-linear
  systems and neural networks.
\newblock {\em Applied and Computational Harmonic Analysis}, 59:85--116, 2022.

\bibitem{MSM22}
Behnam Mafakheri, Iman Shames, and Jonathan Manton.
\newblock First order online optimisation using forward gradients under
  polyak-$\{$$\backslash$L$\}$ ojasiewicz condition.
\newblock {\em arXiv preprint arXiv:2211.15825}, 2022.

\bibitem{MRY13}
Martin Mevissen, Emanuele Ragnoli, and Jia~Yuan Yu.
\newblock Data-driven distributionally robust polynomial optimization.
\newblock In {\em Advances in Neural Information Processing Systems}, pages
  37--45, 2013.

\bibitem{MSJ+16}
Aryan Mokhtari, Shahin Shahrampour, Ali Jadbabaie, and Alejandro Ribeiro.
\newblock Online optimization in dynamic environments: Improved regret rates
  for strongly convex problems.
\newblock In {\em Proceedings of the 55th IEEE Conference on Decision and
  Control (CDC 2016)}, pages 7195--7201, 2016.

\bibitem{PW07}
Georg Pflug and David Wozabal.
\newblock Ambiguity in portfolio selection.
\newblock {\em Quantitative Finance}, 7(4):435--442, 2007.

\bibitem{Pflug12}
Georg~Ch Pflug.
\newblock {\em Optimization of stochastic models: the interface between
  simulation and optimization}, volume 373.
\newblock Springer Science \& Business Media, 2012.

\bibitem{PDM16}
Krzysztof Postek, Dick den Hertog, and Bertrand Melenberg.
\newblock Computationally tractable counterparts of distributionally robust
  constraints on risk measures.
\newblock {\em SIAM Review}, 58(4):603--650, 2016.

\bibitem{SF20}
Iman Shames and Farhad Farokhi.
\newblock Online stochastic convex optimization: Wasserstein distance
  variation.
\newblock {\em arXiv preprint arXiv:2006.01397}, 2020.

\bibitem{SDR21}
Alexander Shapiro, Darinka Dentcheva, and Andrzej Ruszczynski.
\newblock {\em Lectures on stochastic programming: modeling and theory}.
\newblock SIAM, 2021.

\bibitem{SND17}
Aman Sinha, Hongseok Namkoong, and John Duchi.
\newblock Certifiable distributional robustness with principled adversarial
  training.
\newblock In {\em Proceedings of the Machine Learning and Computer Security
  Workshop (co-located with Conference on Neural Information Processing Systems
  2017)}, volume~2, 2017.

\bibitem{SY20}
Tasuku Soma and Yuichi Yoshida.
\newblock Statistical learning with conditional value at risk.
\newblock {\em arXiv preprint arXiv:2002.05826}, 2020.

\bibitem{thick19}
John Thickstun.
\newblock Kantorovich-rubinstein duality, 2019.

\bibitem{TVZ08}
Nikolas Topaloglou, Hercules Vladimirou, and Stavros~A Zenios.
\newblock A dynamic stochastic programming model for international portfolio
  management.
\newblock {\em European Journal of Operational Research}, 185(3):1501--1524,
  2008.

\bibitem{Ury95}
Stanislav Uryasev.
\newblock Derivatives of probability functions and some applications.
\newblock {\em Annals of Operations Research}, 56:287--311, 1995.

\bibitem{Ury00}
Stanislav Uryasev.
\newblock Conditional value-at-risk: Optimization algorithms and applications.
\newblock In {\em proceedings of the IEEE/IAFE/INFORMS 2000 conference on
  computational intelligence for financial engineering (CIFEr)(Cat. No.
  00TH8520)}, pages 49--57. IEEE, 2000.

\bibitem{WZ20}
Xiao Wang and Hongchao Zhang.
\newblock Inexact proximal stochastic second-order methods for nonconvex
  composite optimization.
\newblock {\em Optimization Methods and Software}, 35(4):808--835, 2020.

\bibitem{WBD21}
Killian Wood, Gianluca Bianchin, and Emiliano Dall’Anese.
\newblock Online projected gradient descent for stochastic optimization with
  decision-dependent distributions.
\newblock {\em IEEE Control Systems Letters}, 6:1646--1651, 2021.

\bibitem{woza12}
David Wozabal.
\newblock A framework for optimization under ambiguity.
\newblock {\em Annals of Operations Research}, 193(1):21--47, 2012.

\bibitem{YKRW11}
Sheena Yau, Roy~H Kwon, J~Scott Rogers, and Desheng Wu.
\newblock Financial and operational decisions in the electricity sector:
  Contract portfolio optimization with the conditional value-at-risk criterion.
\newblock {\em International Journal of Production Economics}, 134(1):67--77,
  2011.

\bibitem{ZYY+16}
Lijun Zhang, Tianbao Yang, Jinfeng Yi, Rong Jin, and Zhi-Hua Zhou.
\newblock Improved dynamic regret for non-degenerate functions.
\newblock In {\em Advances in Neural Information Processing Systems}, pages
  732--741, 2017.

\bibitem{ZV14}
William~T Ziemba and Raymond~G Vickson.
\newblock {\em Stochastic optimization models in finance}.
\newblock Academic Press, 2014.

\bibitem{Zink03}
Martin Zinkevich.
\newblock Online convex programming and generalized infinitesimal gradient
  ascent.
\newblock In {\em Proceedings of the 20th International Conference on
  International Conference on Machine Learning (ICML 2003)}, pages 928--935,
  2003.

\end{thebibliography}

\end{document}